\documentclass[11pt]{amsart}
\reversemarginpar  
\usepackage{eucal}
\usepackage{amsthm}
\usepackage{amsfonts}
\usepackage{amsmath}
\usepackage{amssymb}
\usepackage{mathrsfs}
\usepackage{epsfig}
\usepackage{cases}
\newtheorem{theorem}{\bf Theorem}[section]
\newtheorem{proposition}[theorem]{\bf Proposition}
\newtheorem{corollary}[theorem]{\bf Corollary}
\newtheorem{lemma}[theorem]{\bf Lemma}
\newtheorem{notation}[theorem]{\bf Notation}

\newtheorem{question}[theorem]{\bf Question}
\newtheorem{conjecture}[theorem]{\bf Conjecture}

\newtheorem{remark}[theorem]{\bf Remark}
\newtheorem{definition}[theorem]{\bf Definition}

\newcommand{\overbar}[1]{\mkern 1.5mu\overline{\mkern-1.5mu#1\mkern-1.5mu}\mkern 1.5mu}

\newenvironment{psmallmatrix}
 {\left(\begin{smallmatrix}}
 {\end{smallmatrix}\right)}

\newcommand{\z}{\mathbb Z_+}

\newcommand{\hl }{\mathcal{H}}
\newcommand{{\ran}}{\mbox{\rm ran}~}

\newcommand{\hll}{{\mathcal H}^{(\lambda)}}
\newcommand{\hlm}{{\mathcal H}^{(\mu)}}

\newcommand{\jetK}{J_n(K_1,K_2)_{|\rm res\, \Delta}}
\newcommand{\mb}{\mbox{\rm M\"ob}}
\newcommand{\ft}{\varphi_{\theta, a}}
\newcommand{\hk}{(\mathcal H, K)}
\newcommand{\vphi}{\varphi}

\newcommand{\Cphn}{C_{\varphi^{-1}}}
\newcommand {\D} {\mathbb{D}}

\usepackage{mathrsfs}

\numberwithin{equation}{section}

\begin{document}



\title [Orbit of a bounded operator 
under the M\"{o}bius group modulo similarity ]{The orbit of a bounded operator 
under the M\"{o}bius group modulo similarity equivalence}
\author[S. Ghara]{Soumitra Ghara}
\address[S. Ghara]{Department of Mathematics, Indian Institute of Science,
Bangalore 560012, India} \email{ghara90@gmail.com}


\thanks{This work was supported by CSIR SPM Fellowship  (Ref. No. SPM-07/079(0242)/2016-EMR-I). The results in this paper are from the PhD thesis of the author submitted to Indian Institute of Science.}

\date{}
\begin{abstract}
Let M\"{o}b denote the group of biholomorphic automorphisms of the unit disc and  $(\mbox{M\"{o}b} \cdot T)$ be the orbit of a Hilbert space operator $T$ under the action of $\mb$.
If the quotient $(\mbox{M\"{o}b} \cdot T)/\sim$, where $\sim$ is the similarity between two operators is a singleton, then the operator $T$ is said to be weakly homogeneous. 
In this paper, we obtain a criterion to determine if the operator $M_z$ of multiplication by the coordinate function $z$  on a reproducing 
kernel Hilbert space  is weakly homogeneous. 
We use this to show that 
there exists a M\"{o}bius bounded weakly homogeneous 
operator which is not similar to 
any homogeneous operator, answering a question of Bagchi and Misra 
in the negative. Some necessary conditions for  the M\"obius boundedness of a weighted shift are also obtained. As a consequence, it is shown that the Dirichlet shift is not 
M\"{o}bius bounded. 
\end{abstract}

 \renewcommand{\thefootnote}{}
%
%
%

 \maketitle

\section{Introduction}
Let $\mathcal H$ be a complex separable Hilbert space, and let $B(\mathcal H)$ denote the space of bounded linear operators on $\mathcal H$. 
Let M\"ob denote the M\"obius group $ \{\varphi_{\theta,a}: \theta \in [0,2\pi),\, a \in \mathbb D \}$ of all biholomorphic automorphisms  of the 
unit disc $\mathbb D$, where  
$\varphi_{\theta, a}(z) = e^{i \theta} \tfrac{z-a}{1-\overbar{a}z},\;\;z\in \mathbb D.$ Clearly, it is a topological group with the topology induced by $\mathbb T\times \D$.

For an operator $T \in B(\mathcal H)$ with $\sigma(T)\subseteq \bar {\mathbb D}$, define the operator $\varphi(T)$, $\varphi \in \mb$, using  the functional calculus valid for  functions holomorphic in a neighbourhood of $\sigma(T)$. An operator $T$ with $\sigma(T)\subseteq \bar {\mathbb D}$ is said to be homogeneous  if $\varphi(T)$ is unitarily equivalent to $T$ for all $\varphi\in \mb$.  These operators have been studied extensively in the recent years (\cite{constantchar, onhomocontraction, homogeneoussurvey, homoshift}).
It follows from the spectral mapping theorem that the spectrum of a homogeneous operator is invariant under the action of M\"{o}b,
and therefore, it is either the unit circle $\mathbb T$ or the closed unit disc $\bar{\D}$. In this paper, we study a much larger class of operators, namely, the ones that are weakly homogeneous (cf. \cite{homogeneoussurvey}, \cite{onhomocontraction}). 
\begin{definition}
An operator $T \in B(\mathcal H)$ is said to be weakly 
homogeneous if $\sigma(T)\subseteq \bar {\mathbb D}$ and $\varphi(T)$ is similar to $T$ 
for all $\varphi$ in $\mbox{\rm {M\"ob}}$. 
\end{definition}


For two operators $T_1$ and $T_2$ in $B(\mathcal H)$, we write $T_1\sim T_2$ if there exists an invertible linear operator $L$ such that $L T_1 L^{-1} = T_2$. Clearly, $\sim$ defines an equivalence relation on $B(\mathcal H)$. Weakly homogeneous operators can also be defined in the following alternative way. An operator $T$ in $B(\mathcal H)$
with $\sigma(T)\subseteq \bar {\mathbb D}$ is weakly homogeneous if the quotient space $(\mb \cdot T)/\sim$ is a singleton, where $\mb \cdot T$ is the orbit of the operator $T$ under the natural action of the M\"{o}bius group. Replacing the invertible operator by a unitary we get a different equivalence relation which gives rise to homogeneous operators.

As in the case of a homogeneous operator, the spectrum of a weakly homogeneous operator is also either $\mathbb T$ or $\overbar{\D}$.  It is easy to see that an operator $T$ is weakly homogeneous if and only if  $T^*$ is weakly homogeneous.
Since two normal operators are similar if and only if they are unitarily equivalent, it follows that a normal operator $N$ is homogeneous if and only if it is weakly homogeneous.

It is not hard to verify that an operator $T$ which is similar to a homogeneous operator is weakly homogeneous. Indeed, if 
$T=XSX^{-1}$ for some homogeneous operator $S$ and an invertible  operator $X$, 
then we have
$\varphi(T)=X\varphi(S)X^{-1}$. Since $S$ is homogeneous, 
$\varphi(S)=U_\varphi S U_{\varphi}^*$ for some unitary operator $U_{\varphi}$. Hence 
\begin{equation}\label{eqnsimtohomo}
\varphi(T)=(XU_\varphi) S (XU_\varphi)^{-1}=(XU_\varphi X^{-1})T(XU_\varphi X^{-1})^{-1}.
\end{equation}
is weakly homogeneous. 
Hence every operator which is similar to a homogeneous operator is weakly homogeneous. The converse of this statement is not true, that is, a weakly homogeneous operator need not be similar to any homogeneous operator. To see this, it would be useful to recall the definition of a M\"obius bounded operator.

M\"obius bounded operators were introduced in \cite{Shields} by Shields. An operator $T$ on a Banach space $\mathcal B$ is said to be M\"obius bounded if $\sigma(T)\subseteq \overbar{\mathbb D}$ and $\sup_{\varphi\in \;\mbox{\rm M\"ob}}\|\varphi(T)\|$ is finite. We will only discuss M\"obius bounded operators on Hilbert spaces. By the von Neumann's inequality, every contraction on a Hilbert space is M\"obius bounded. Also, if $T$ is an operator which is similar to a  homogeneous operator, then by \eqref{eqnsimtohomo} it is easily verified that $T$ is M\"obius bounded. In \cite{homogeneoussurvey}, the existence of a weakly homogeneous operator which is not M\"obius bounded was given. Hence it cannot be similar to any homogeneous operator. In the same paper, the following question was raised.

\begin{question}[{ \cite[Question 10]{homogeneoussurvey}}]
\label{quesweakhomo}
Is it true that every M\"obius bounded weakly homogeneous operator is similar to a homogeneous operator?
\end{question}

One of the main results of this paper is Theorem \ref{thmexample} which gives a family of M\"obius bounded weakly homogeneous operators in $B_1(\D)$ which are not similar to any homogeneous operator.

Now we recall some basic properties of reproducing kernel Hilbert spaces.
Let $\Omega\subset \mathbb C$ be a bounded domain, and let 
$M_n(\mathbb{C})$, $n\geq 1$, denote the space of all $n\times n$ complex matrices. A function $ K:\Omega \times \Omega\to \mathcal M_n(\mathbb{C})$ 
is said to be a non-negative definite kernel  
if for any subset $\{w_1,\ldots,w_l\}$ of $\Omega$,  
the $l \times l$ block matrix $\big(K(w_i,w_j)\big)_{i,j=1}^{l}$ is non-negative definite, that is,
$\sum_{i,j=1}^l\left\langle K(w_i,w_j)\eta_j,\eta_i\right\rangle \geq 0$
for all $\eta_1,\ldots,\eta_l\in \mathbb C^n.$ 
If, in addition, $\big(K(w_i,w_j)\big)_{i,j=1}^{n}$
is also invertible, we say that $K$ is a positive definite kernel.
We always assume that  $K(z,w)$ is sesqui-analytic, that is, it is holomorphic in $z$ and anti-holomorphic in $w$. If $K:\Omega\times \Omega\to{\mathcal M}_n(\mathbb{C})$ is a non-negative definite kernel, then there exists a unique Hilbert space $\hl$ consisting of $\mathbb C^n$
valued holomorphic functions on $\Omega$ such that for all $w$  in $\Omega$ and $\eta$ in ${\mathbb C}^n$, the function $K(\cdot,w)\eta$ is in $\mathcal H $, where $K(\cdot,w)\eta z=K(z,w)\eta$, $z\in \Omega$, and 
for all $f$ in $\mathcal{H}$,
$\left \langle f, K(\cdot,w)\eta\right\rangle_\mathcal{H}=\left\langle f(w),\eta\right\rangle_{\mathbb C^n}$. 
We let $(\mathcal H, K)$ denote the unique reproducing kernel Hilbert space $\mathcal H$ determined by the non-negative definite kernel $K$. 
We refer to \cite{Aro}, \cite{pick-int} and \cite{PaulsenRaghupati} for the relationship between non-negative definite kernels and Hilbert spaces with the reproducing property as above.

We now discuss an important class of operators introduced by Cowen and Douglas in the very influential paper \cite{CD}. Let  $\Omega\subset \mathbb C$ be a bounded domain, and let $k$ be a positive integer. A bounded operator $T$ is said to be in the Cowen-Douglas class $B_k(\Omega)$ if $T$ satisfies the following requirements: 
\begin{itemize}
\item[\rm(i)] 
\rm dim $\ker (T-w) = k,~~w\in\Omega$
\item[\rm(ii)]
$\ran (T-w)$ is closed for all $w\in\Omega$
\item[\rm(iii)]
$\overbar \bigvee \big\{\ker (T-w): w\in \Omega \big\}=\mathcal H.$ 
\end{itemize} 

Every $\boldsymbol T\in B_k(\Omega)$ corresponds to a rank $k$
holomorphic hermitian vector bundle $E_{\boldsymbol T}$ defined by 
$$E_{\boldsymbol T}=\{(w,x)\in \Omega\times \mathcal H:x\in \ker (T-w)\}$$ 
and $\pi(w,x)=w$, $(w,x)\in E_{\boldsymbol T}$. 
It is known that if $T$ is an operator in $B_k(\Omega)$,
then $T$ is unitarily equivalent to the adjoint $M_z^*$ of the multiplication operator by the coordinate function $z$ on some 
reproducing kernel Hilbert space $(\mathcal H, K)$ consisting of $\mathbb C^k$ valued holomorphic functions on $\Omega^*$, 
where $\Omega^*=\{z:\bar{z}\in \Omega\}.$  

For $\lambda>0$, let $K^{(\lambda)}$ denote the positive definite kernel $(1-z\bar{w})^{-\lambda}$ on $\D\times \D$, and
whenever $K$ is equal to $K^{(\lambda)}$, we write $\hll$ instead of $(\hl, K^{(\lambda)})$.
It known that the adjoint $M_z^*$ of the multiplication operator by the coordinate function $z$ on $\hll$, $\lambda>0$, is homogeneous and  
 upto unitary equivalence, every homogeneous operator in $B_1(\mathbb D)$ is of this form,
 see \cite{curvandbackwardshift}.

Now we come back to the discussion on M\"{o}bius bounded operators. Recall that an operator $T$ on a Hilbert space $\hl$ is said to be power bounded if $\sup_{n\geq 0}\|T^n\| < \infty$. In \cite{Shields}, it was shown that every power bounded operator is M\"{o}bius bounded.  An example of an operator on a Banach space which is M\"{o}bius bounded but not power bounded was also given in that paper.  
The multiplication operator $M_z$ on the Hilbert space $(\hl , K^{(\lambda)}),$ $0<\lambda < 1,$ is homogeneous, therefore, M\"{o}bius bounded, however, it is not power bounded. This was noted in \cite{homogeneoussurvey}.  
Although a M\"{o}bius bounded operator need not be power bounded, Shields proved that if $T$ is a 
M\"{o}bius bounded operator on a Banach space, then there exists a $c>0$ such that  $\|T^n\|\leq c (n+1)$ for all $n\in \z$. But for operators on Hilbert spaces, he made the following conjecture.

\begin{conjecture}[Shields,\cite{Mobiusbounded}]
If $T$ is a M\"obius bounded operator on a Hilbert space, then there exists a constant $c>0$ such that $\|T^n\|\leq c(n+1)^{\frac{1}{2}}$ for all $n\in \z$.
\end{conjecture}

This conjecture is verified for the class of quasi-homogeneous operators introduced recently in the paper \cite{quasihomogeneous}.

The paper is organized as follows. In section 2,  we obtain a necessary and sufficient condition to determine if the multiplication operator on a reproducing kernel Hilbert space 
$(\hl, K)$ is weakly homogeneous. We use this to construct new weakly homogeneous operators starting from a pair of weakly homogeneous operators. In section 3, we study weakly homogeneous operators in a the class $\mathcal FB_2(\D) \subseteq B_2(\D)$. M\"obius bounded operators are discussed in section 4. Finally, in section 5, we show that there exists a continuum of weakly homogeneous operators which are not similar to any homogeneous operator.
\section{Jet construction, weighted composition operators and weak homogeneity}\label{secjetweak}
Throughout this section, we assume that $\Omega\subset \mathbb C$ is a bounded domain. Let ${\rm Hol}(\Omega)$ denote the space of all scalar valued holomorphic functions on $\Omega$, and let ${\rm Hol}(\Omega, \tilde{\Omega})$ denote the space of all holomorphic functions on $\Omega$ taking values in  $\tilde{\Omega}$, where 
$\tilde{\Omega}\subset \mathbb C^r$, $r\geq 1$.


%
Let $\psi\in {\rm Hol}(\Omega, \mathcal M_n(\mathbb C))$. 
Let $M_{\psi}$ denote the linear map on ${\rm Hol}(\Omega, \mathbb C^n)$ defined by point-wise multiplication:
$$(M_{\psi}f)(\cdot)=\psi(\cdot)f(\cdot),f\in {\rm Hol}(\Omega, \mathbb C^n).$$
For a holomorphic self map $\varphi$ of $\Omega$, let 
$C_{\varphi}$ denote the linear map on ${\rm Hol}(\Omega, \mathbb C^n)$ defined by composition:
$$(C_\varphi f)(\cdot)=(f\circ\varphi)(\cdot), f\in {\rm Hol}(\Omega, \mathbb C^n).$$ If $K:\Omega\times\Omega\to\mathcal M_n(\mathbb C)$ is a non-negative definite kernel, then, in general, neither $M_{\psi}$  nor $C_\varphi$ maps $(\mathcal H, K)$ into  $(\mathcal H, K)$. However, by the Closed graph theorem, 
whenever either one (or both) of them maps $(\mathcal H, K)$ into  $(\mathcal H, K)$, then it is (they are)  bounded. Whenever the map $M_{\psi} C_{\varphi}$ is 
bounded on $(\mathcal H, K)$, it is called a weighted composition operator on 
$(\mathcal H, K)$.  For $w\in \Omega$, $\eta\in \mathbb C^n$ and $h\in (\hl, K)$, we see that 
\begin{align*}
\left\langle (M_{\psi} C_{\varphi})^*K(\cdot,w)\eta, h\right\rangle
&=\left\langle K(\cdot,w)\eta, \psi(\cdot)h(\varphi(\cdot))\right\rangle\\
&=\left\langle \eta, \psi(w)h(\varphi(w))\right\rangle\\
&=\left\langle \psi(w)^*\eta, h(\varphi(w))\right\rangle\\
&=\left\langle K\big(\cdot,\varphi(w)\big)\big(\psi(w)^*\eta\big), h\right\rangle.
\end{align*}
Therefore
\begin{equation}\label{eqnwtcomkernelfns}
(M_{\psi} C_{\varphi})^*K(\cdot,w)\eta=K(\cdot,\varphi(w))(\psi(w)^*\eta),\;\;w\in \Omega,\eta\in \mathbb C^n.
\end{equation}
 
We now  recall the jet construction. Suppose that $K_1, K_2:\Omega\times\Omega\to \mathbb C$ are two non-negative definite kernels. Then  $(\hl, K_1)\otimes(\hl, K_2)$ is a reproducing kernel Hilbert space with the reproducing kernel $K_1\otimes K_2$, where $K_1\otimes K_2:(\Omega\times \Omega)\times(\Omega\times \Omega)\to \mathbb C$ is given by 
$$(K_1\otimes K_2)(z,\zeta;w,\rho)=K_1(z,w)K_2(\zeta,\rho),\; \; z,\zeta,w,\rho\in \Omega.$$ 
We also make the standing assumption that the multiplication operator  
$M_z$  on $(\mathcal H, K_1)$ as well as on $(\mathcal H, K_2)$ is bounded. For $n\in \z$, let $\mathcal A_n$ be the subspace of $(\hl, K_1)\otimes (\hl, K_2)$ given by 
\begin{equation}\label{eqnwtcompadj}
\mathcal A_n:=\big \{f\in (\mathcal H, K_1)\otimes (\mathcal H, K_2): \big(\big(\tfrac{\partial}{\partial \zeta}\big)^i f(z,\zeta)\big)_{|\Delta}=0,\;0
\leq i\leq n\big\},
\end{equation}
where $\Delta$ is the diagonal set $\{(z,z):z\in \Omega\}$.
Let $J_n:(\hl, K_1)\otimes (\hl, K_2)\to {\rm Hol}(\Omega\times \Omega, \mathbb C^{n+1})$ be the map given by the following formula
$$(J_nf)(z,\zeta)=\sum_{i=0}^n\big(\tfrac{\partial}{\partial \zeta}\big)^{i} 
f(z,\zeta)\otimes e_{i},~f\in (\mathcal H ,K_1)\otimes (\mathcal H, K_2),$$
where $\big \{e_{i}\big \}_{i=0}^n$
is the standard orthonormal basis of $\mathbb C^{n+1}$. Also let $R:\ran J_n\to \rm Hol(\Omega, \mathbb{C}^{n+1})$ be the restriction map, that is, $R(\mathbf{h})=\mathbf{h}_{|\Delta},~\mathbf {h}\in \ran J_n$.

Clearly, $\ker R J_n=\mathcal A_n$. Hence the map $R J_n: \mathcal A_n^\perp \to \rm Hol(\Omega, \mathbb{C}^{n+1})$ is one to one. Therefore we can give a natural inner product on $ \ran {R J_n} $, namely, 
\begin{equation}\label{eqninnp}
\langle RJ_n(f),R J_n(g)\rangle=\langle P_{\mathcal A_n^\perp}f,P_{\mathcal A_n^\perp}g\rangle,~f,g\in (\hl,K_1)\otimes (\hl,K_2).
\end{equation}
In what follows, we think of $\ran {R J_n}$ as a Hilbert space equipped with this inner product. From \eqref{eqninnp}, it is clear that the map ${R J_n}_{|\mathcal A_n^\perp}:\mathcal A_n^\perp\to \ran {R J_n}$ is unitary. 
The theorem stated below is a straightforward generalization of one of the main results from \cite{D-M-V}. 
\begin{theorem}$(${\rm cf. {\cite [Proposition 2.3]{D-M-V}}}$)$\label{thmjetcons}
Let $K_1,K_2:\Omega\times\Omega\to \mathbb C$ be  two non-negative definite kernels. Then $\ran RJ_n$ is a reproducing kernel Hilbert space and its  reproducing kernel $J_n(K_1,K_2)_{|\rm res \, \Delta}$ is given by the formula 

$$J_n(K_1,K_2)_{|\rm res \, \Delta}(z,w):=\Big(K_1(z,w)\big(\tfrac{\partial}{\partial z}\big)^{i}\big(\tfrac{\partial}{\partial \bar{w}}\big)^{j} K_2(z,w)\Big)_{ i, j=0}^{n},\;\;z,w\in \Omega.$$
Moreover, the multiplication operator $M_z$ on $(\hl, J_n(K_1,K_2)_{|\rm res \, \Delta})$  is unitarily equivalent to the operator 
$P_{\mathcal A_n^\perp}(M_z\otimes I)_{|\mathcal A_n^\perp}$ via the unitary ${RJ_n}_{|\mathcal A_n^\perp}$. 
\end{theorem} 

For any $\psi\in \rm Hol(\Omega)$, let 
$\psi^{(i)}(z)$, $i\in \mathbb Z_+$, denote the $i$th derivative of $\psi$ at the point $z$, and let $(\mathcal J_n\psi)(z)$, $z\in \Omega$, denote the $(n+1)\times (n+1)$ lower triangular matrix given by 

\begin{equation*}
\begingroup
\renewcommand*{\arraystretch}{1.5}
(\mathcal J_n\psi)(z):=
\begin{pmatrix}
\psi(z)&0 & 0&\hdots&0\\
\tbinom{1}{0}\psi^{(1)}(z)&\psi(z)& 0&\hdots &0\\
\binom{2}{0}\psi^{(2)}(z)&\binom{2}{1}\psi^{(1)}(z)&\psi(z) & \hdots&0\\
\vdots&\vdots&\vdots &\vdots&\vdots\\
\binom{n}{0}\psi^{(n)}(z)&\hdots&\hdots&\tbinom{n}{n-1}\psi^{(1)}(z)&\psi(z)
\end{pmatrix}.
\endgroup
\end{equation*}

Recall that for  $f\in \mbox{\rm {Hol}}(\Omega)$ and $\varphi\in \mbox{\rm Hol}(\Omega, \Omega)$, the Fa\`a di Bruno's formula (cf. \cite[page 139]{DiBruno}) for the $i$th derivative of the composition function $f \circ \varphi$ is the following:
\begin{equation}\label{dibruno}
(f \circ \varphi)^{(i)}(z)=
\sum_{j=1}^i f^{(j)}\big(\varphi(z)\big)B_{i,j}
\big(\varphi^{(1)}(z),...,\varphi^{(i-j+1)}(z)\big), \;\;z\in \Omega,
\end{equation}
where $B_{i,j}(z_1,\ldots,z_{i-j+1})$, $i\geq j\geq 1$, are the Bell's polynomials. 
Furthermore, let $(\mathcal B_n\varphi)(z)$ denote the $(n+1)\times (n+1)$  lower triangular matrix of the form 
\renewcommand*{\arraystretch}{1.5}
\[
(\mathcal B_n\varphi)(z):=\left(
\begin{array}{c c}
1 & \boldsymbol 0 \\ 
\boldsymbol{0} & \big (\!\! \big ( B_{i,j}\big(\varphi^{(1)}(z),...,\varphi^{(i-j+1)}(z)\big) \big )\!\!\big)_{i,j=1}^n 
\end{array}
\right), z\in \Omega,
\]
where $B_{i,j}$, $1 \leq i < j \leq n$, is set to be $0$.  
\renewcommand*{\arraystretch}{1}
%
%
%

One of the main results of this subsection is the theorem below identifying the compression of the tensor product of two weighted composition operators 
with another weighted composition operator. 

\begin{theorem}\label{wtcompthm1}
Let $\psi_1,\psi_2\in \mbox{\rm Hol}(\Omega)$ and let  $\varphi \in \mbox{\rm Hol}(\Omega,\Omega)$. Suppose that the weighted  composition operators $ M_{\psi_1} C_{\varphi}$ on $ (\mathcal H, K_1) $ 
and $ M_{\psi_2} C_{\varphi} $ on $ (\mathcal H, K_2) $ are bounded.
Then  $P_{\mathcal A_n^\perp}
(M_{\psi_1} C_{\varphi}\otimes M_{\psi_2} C_{\varphi})
_{|\mathcal A_n^\perp}$ is unitarily equivalent to the operator 
$ M_{\psi_1 (\mathcal J_n\psi_2) (\mathcal B_n\varphi )}C_{\varphi} $ on $\big(\mathcal H , \jetK \big).$

In particular,
the operator $ M_{\psi_1(\mathcal J_n\psi_2)(\mathcal B_n\varphi )}C_{\varphi} $ is bounded on $\big(\;\mathcal H,J_n(K_1,K_2)_{|\rm res\, \Delta}\;\big)$ and 
$$
\| M_{\psi_1 (\mathcal J_n\psi_2) (\mathcal B_n\varphi )}C_{\varphi} \|
\leq \|M_{\psi_1} C_{\varphi}\|\|M_{\psi_2} C_{\varphi}\|.
$$
\end{theorem}
Before, we give the proof of Theorem \ref{wtcompthm1}, we state a second theorem refining some of the statements in it.  
Let $\rm{Aut}(\Omega)$ denote the group of biholomorphic automorphisms of the domain $\Omega$. 

\begin{theorem}\label{wtcompthm2}
Let  $\psi_1,\psi_2\in \mbox{\rm Hol}(\Omega)$ and $\varphi \in \mbox{\rm Aut}(\Omega)$. 
\begin{itemize}
\item[\rm (i)]If the operators 
$M_{\psi_1}C_{\varphi}$ on $(\hl, K_1)$ and $M_{\psi_2}C_{\varphi}$ on $(\hl, K_2)$ are bounded and invertible, then so is the operator $ M_{\psi_1(\mathcal J_n\psi_2)(\mathcal B_n\varphi )}C_{\varphi} $ on $\big(\mathcal H,J_n(K_1,K_2)_{|\rm res\, \Delta}\big)$.
\item[\rm (ii)] 
If the operators $M_{\psi_1}C_{\varphi}$ on $(\hl, K_1)$ and $M_{\psi_2}C_{\varphi}$ on $(\hl, K_2)$ are unitary, then so is the operator $ M_{\psi_1(\mathcal J_n\psi_2)(\mathcal B_n\varphi )}C_{\varphi} $ on $\big(\mathcal H, J_n(K_1,K_2)_{|\rm res\, \Delta}\big)$.
\end{itemize}
\end{theorem}

The following lemma is an essential tool in  the proof of Theorem \ref{wtcompthm2}. However, the straightforward proof is omitted.
\begin{lemma}\label{lemwtcomp}
Let $H$ be a Hilbert space, and let $X:H\to H$
be a bounded, invertible operator. Suppose that $H_0$ is a closed subspace of 
$ H$ which is invariant under both $X$ and $X^{-1}.$ Then the operators $X_{|H_0}$ and $P_{H_0^\perp}X_{|H_0^\perp}$ are invertible. Moreover, if $X$ is unitary, then
$H_0^\perp$ is also invariant under $X$, and 
the operators $X_{|H_0}$ and $X_{|H_0^\perp}$
are unitary.
\end{lemma}

\textbf{Proof of the Theorem} $\boldsymbol {\ref{wtcompthm1}}$.
First, set 
$$(\psi_1\otimes\psi_2)(z,\zeta):=\psi_1(z)\psi_2(\zeta) \;\;\;\mbox{and} 
\;\;\;\boldsymbol{\varphi}(z,\zeta):=(\varphi(z),\varphi(\zeta)),\;\;\; z,\zeta\in \Omega.$$ 
Clearly, $\psi_1\otimes\psi_2\in \mbox{\rm Hol}(\Omega\times\Omega)$ and $\boldsymbol{\varphi}\in \mbox{\rm Hol}(\Omega\times\Omega, \Omega\times\Omega)$. 
Using \eqref{eqnwtcomkernelfns}, it is easily verified that 
\begin{equation}\label{eqnwtcomtensor}
M_{\psi_1}C_{\varphi}\otimes M_{\psi_2} C_{\varphi}=M_{\psi_1\otimes\psi_2}C_{\boldsymbol{\varphi}}.
\end{equation}
By Theorem \ref{thmjetcons}, 
we will be done if we can show that
\begin{equation}\label{eqnwtcompthm1}
\big((RJ_n)_{|\mathcal A_n^\perp}\big) P_{\mathcal A_n^\perp}(M_{\psi_1} 
C_{\varphi}\otimes M_{\psi_2} C_{\varphi})
_{|\mathcal A_n^\perp}\big((RJ_n)_{|\mathcal A_n^\perp}\big)^* = M_{\psi_1 (\mathcal J_n\psi_2) (\mathcal B_n\varphi )}C_{\varphi}.
\end{equation}
To verify this, let $(RJ_n)(f)$ be an arbitrary vector in $\big(\hl, \jetK\big)$ where $f\in \mathcal A_n^\perp$. 
Since $\ker RJ_n=\mathcal A_n$, it follows that
\begin{equation}\label{eqn_weakhom}
\big((RJ_n)_{|\mathcal A_n^\perp}\big) P_{\mathcal A_n^\perp}(M_{\psi_1} 
C_{\varphi}\otimes M_{\psi_2} C_{\varphi})
_{|\mathcal A_n^\perp}\big((RJ_n)_{|\mathcal A_n^\perp}\big)^*(RJ_n f)=
(RJ_n)(M_{\psi_1} C_{\varphi}\otimes M_{\psi_2} C_{\varphi}\big)(f).
\end{equation}
Using \eqref{eqnwtcomtensor}, we see that
\begin{align}\label{egnwtcom1}
\begin{split}
(RJ_n)(M_{\psi_1} C_{\varphi}\otimes M_{\psi_2} C_{\varphi}\big)(f)&=(RJ_n)\big(\psi_1(z)\psi_2(\zeta) f(\varphi(z),\varphi(\zeta)\big)\\
&=\sum_{i=0}^n \Big (\big(\tfrac{\partial}{\partial \zeta}\big)^i\psi_1(z)\psi_2(\zeta)f(\varphi(z),\varphi(\zeta))\Big)_{|_{\Delta}}\otimes e_i.
\end{split}
\end{align}

Also a straightforward computation, noting that $\mathcal J_n \psi_2$ and $\mathcal B_n\varphi$ are lower triangular, shows that
\begin{align}\label{egnwtcom2}
\begin{split}
&\big(M_{\psi_1 (\mathcal J_n\psi_2) (\mathcal B_n\varphi )}C_{\varphi}\big)\big((RJ_n)f\big)(z)\\
&\quad\quad\quad=\big(M_{\psi_1 (\mathcal J_n\psi_2) (\mathcal B_n\varphi )}C_{\varphi}\big)\big(\textstyle  
\sum_{i=0}^n\big(\big(\tfrac{\partial}{\partial \zeta}\big)^{i} 
f(z,\zeta)\big)_{|\Delta}\otimes e_{i}
\big)\\
&\quad\quad\quad=\psi_1(z)\textstyle\sum_{i=0}^n \Big(\textstyle\sum_{j=0}^i\big((\mathcal J_n\psi_2) 
(\mathcal B_n\varphi)\big)_{i,j}(z)\big(\big(\tfrac{\partial}{\partial \zeta}\big)^{j} 
f(z,\zeta)\big)_{|\Delta}(\varphi(z),\varphi(z))\Big)
\otimes e_{i}.
\end{split}
\end{align}
Hence, in view of \eqref{eqn_weakhom}, \eqref{egnwtcom1} and \eqref{egnwtcom2}, to verify \eqref{eqnwtcompthm1}, it suffices to show that
\begin{align}\label{eqnwtcom3}
\begin{split}
&\Big (\big(\tfrac{\partial}{\partial \zeta}\big)^{i}\psi_2(\zeta)f(\varphi(z),\varphi(\zeta))\Big)(z,z)\\
&\quad\quad\quad\quad\quad=\textstyle\sum_{j=0}^i\big((\mathcal J_n\psi_2) 
(\mathcal B_n\varphi)\big)_{i,j}(z)\big(\big(\tfrac{\partial}{\partial \zeta}\big)^{j} 
f(z,\zeta)\big)(\varphi(z),\varphi(z)),\;
\;0\leq i\leq n.
\end{split}
\end{align}
Since $\big((\mathcal J_n\psi_2)(\mathcal B_n\varphi)\big)_{0,0}(z)=\psi_2(z)$, equality in both sides of \eqref{eqnwtcom3} is easily verified for the case $i=0$.
For $ 1\leq i\leq n$, we see that 
\begin{align}\label{eqnremark}
\begin{split}
&\Big(\big(\tfrac{\partial}{\partial \zeta}\big)^{i}  \big(\psi_2(\zeta)
f(\varphi(z),\varphi(\zeta))\big)\Big)(z,z)\\
& \quad\quad=\Big(\psi_2^{(i)}(\zeta) f(\varphi(z),\varphi(\zeta))+\textstyle\sum_{p=1}^{i}\binom{i}{p}\psi_2^{(i-p)}
(\zeta)\big(\tfrac{\partial}{\partial \zeta}\big)^p\big(f(\varphi(z),\varphi(\zeta)\big)\Big)(z,z)\\
&\quad\quad= \psi_2^{(i)}(z) f(\varphi(z),\varphi(z))\\
& \quad\quad\quad\quad\quad\quad\quad\quad+
\textstyle \sum_{p=1}^i \textstyle\sum_{j=1}^{p}\binom{i}{p}\psi_2^{(i-p)}(z)
 (\mathcal B_n\varphi)_{p,j}(z)\big(\big(\tfrac{\partial}{\partial \zeta}\big)^j f(z,\zeta)\big)
(\varphi(z),\varphi(z)).
\end{split}
\end{align}
Here the first equality follows from the Leibniz rule for derivative of product
while the second one follows from \eqref{dibruno}.
Finally, we compute
\begin{align*}
\begin{split}
&\textstyle\sum_{j=0}^i\big((\mathcal J_n\psi_2) 
(\mathcal B_n\varphi)\big)_{i,j}(z)\big(\big(\tfrac{\partial}{\partial \zeta}\big)^{j} 
f(z,\zeta)\big)(\varphi(z),\varphi(z))\\
&\quad\quad =\textstyle\sum_{j=0}^i \textstyle \sum_{p=j}^i \binom{i}{p}\psi_2^{(i-p)}(z)(\mathcal B_n\varphi)_{p,j}(z)
 \big(\big(\tfrac{\partial}{\partial \zeta}\big)^{j} 
f(z,\zeta)\big)(\varphi(z),\varphi(z))\\
&\quad\quad = \psi_2^{(i)}(z)f(\varphi(z),\varphi(z))+\textstyle\sum_{j=1}^i 
\textstyle \sum_{p=j}^i \binom{i}{p}\psi_2^{(i-p)}(z)(\mathcal B_n\varphi)_{p,j}(z)
\big(\big(\tfrac{\partial}{\partial \zeta}\big)^{j} 
f(z,\zeta)\big)(\varphi(z),\varphi(z))
\\
&\quad\quad=\psi_2^{(i)}(z)f(\varphi(z),\varphi(z))+\textstyle\sum_{p=1}^i 
\textstyle \sum_{j=1}^p \binom{i}{p}\psi_2^{(i-p)}(z)(\mathcal B_n\varphi)_{p,j}(z)
\big(\big(\tfrac{\partial}{\partial \zeta}\big)^{j} 
f(z,\zeta)\big)(\varphi(z),\varphi(z)).
\end{split}
\end{align*}
Here the second equality follows since $(\mathcal B_n\varphi)_{q,0}=\delta_{q0}$,
$0\leq q\leq n$.  

The equality in \eqref{eqnwtcom3} is therefore verified, completing the proof of the theorem.  \hfill $\Box$

\begin{remark}\label{reminvarintsub}
From \eqref{eqnremark}, we see that 
if the hypothesis of Theorem 
\ref{wtcompthm1} is in force, then the subspace $\mathcal A_n$ is invariant under the operator $M_{\psi_1}C_{\varphi}\otimes M_{\psi_2}C_{\varphi}$.
\end{remark}

\textbf{Proof of the Theorem} $\boldsymbol{\ref{wtcompthm2}\; \rm(i)}$.
By hypothesis the operators $M_{\psi_1} C_{\varphi}$ on $(\mathcal H, K_1)$ and $M_{\psi_2}C_{\varphi}$ 
on $(\mathcal  H, K_2)$ are bounded and invertible. It is easy to see that 
$(M_{\psi_i}C_{\varphi})^{-1}= M_{\chi_i} C_{\varphi^{-1}},$ where 
$\chi_i = \frac {1}{\psi_i\circ {\varphi^{-1}}}$, $i=1,2.$ Consequently, $(M_{\psi_1}C_{\varphi} \otimes M_{\psi_2}C_{\varphi})^{-1}=
M_{\chi_1}C_{\varphi^{-1}} \otimes M_{\chi_2}C_{\varphi^{-1}}.$ Therefore, by Remark \ref{reminvarintsub}, $\mathcal A_n$ is invariant under both $M_{\psi_1}C_{\varphi} \otimes M_{\psi_2}C_{\varphi}$ and $(M_{\psi_1}C_{\varphi} \otimes M_{\psi_2}C_{\varphi})^{-1}$.
Hence, by Lemma $\ref{lemwtcomp}$, the operator 
$P_{\mathcal A_n^\perp}(M_{\psi_1} C_{\varphi}\otimes M_{\psi_2} C_{\varphi})
_{|\mathcal A_n^\perp} $ is invertible.
An application of  Theorem $\ref{wtcompthm1}$ now completes the proof.
 \hfill$\Box$

\textbf{Proof of the Theorem} $\boldsymbol{\ref{wtcompthm2}\; \rm(ii)}$.
If $M_{\psi_1}C_{\varphi}$ on $(\hl, K_1)$ and $M_{\psi_2}C_{\varphi}$ on $(\hl,K_2)$  are unitary, then 
so is the operator $M_{\psi_1}C_{\varphi} \otimes M_{\psi_2}C_{\varphi}.$
Hence, by the argument used in part $\rm (i)$ of this theorem together with Lemma $\ref{lemwtcomp}$, we see that $\mathcal A_n$ is reducing under $M_{\psi_1}C_{\varphi} \otimes M_{\psi_2}C_{\varphi}$, and therefore 
$(M_{\psi_1}C_{\varphi} \otimes M_{\psi_2}C_{\varphi})_{|\mathcal A_n^\perp}$
is unitary. Hence, by Theorem $\ref{wtcompthm1}$, we conclude that  $ M_{\psi_1 (\mathcal J_n\psi_2) (\mathcal B_n\varphi )}C_{\varphi} $ on 
$\big(\mathcal H, J_n(K_1,K_2)_{|\rm res\, \Delta}\big)$ is unitary.\hfill $\Box$

Recall that the compression of the operators $M_{z}\otimes I$ and $I \otimes M_z$ acting on $(\mathcal H,K_1)\otimes (\mathcal H,K_2)$ to the subspace $\mathcal A_0^\perp$ are unitarily equivalent to the operator $M_z$ on the Hilbert space $(\mathcal H, K_1 K_2)$.
The following corollary isolates the case of $\mathcal A_0$ from Theorem \ref{wtcompthm1} and Theorem  \ref{wtcompthm2} providing a model for the compression of the tensor product of the weighted composition operators on $(\mathcal H, K_1)$ and $(\mathcal H,K_2)$ to $\mathcal A_0^\perp$ in this particular case.  
\begin{corollary}
Let  $\psi_1,\psi_2\in \mbox{\rm Hol}(\Omega)$ and $\varphi \in \mbox{\rm Hol}(\Omega,\Omega)$. Let $K_1$ and $K_2$ be two scalar valued non-negative definite kernels on $\Omega \times \Omega$. 
Suppose that the weighted composition operators $M_{\psi_1} C_{\varphi}$ on $(\mathcal H, K_1)$
and $ M_{\psi_2} C_{\varphi}$ on $(\mathcal H, K_2)$ are 
bounded. 
Then the operator $ M_{\psi_1 \psi_2}C_{\varphi}$ on $(\mathcal{H},K_1K_2 )$ is 
bounded  with 
$$\| M_{\psi_1 \psi_2 }C_{\varphi} \|\leq \|M_{\psi_1} 
C_{\varphi}\|\|M_{\psi_2} C_{\varphi}\|.$$ 
Moreover, if the operators  $M_{\psi_1}C_{\varphi}$ on $(\hl, K_1)$ and $M_{\psi_2}C_{\varphi}$ on $(\hl, K_2)$ are invertible (resp. unitary), and $\varphi\in {\rm Aut}(\Omega)$,  then the operator $M_{\psi_1\psi_2}C_{\varphi}$ is also invertible (resp. unitary) on $(\hl, K_1K_2)$.
\end{corollary}
\subsection{Weighted composition operators and weakly homogeneous operators}
In this subsection, we find a criteria to determine if the multiplication operator $M_z$ on a reproducing kernel Hilbert space is weakly homogeneous.
We begin with a preparatory lemma. First, recall that a positive definite kernel $K:\D\times \D\to \mathcal M_n(\mathbb C)$ is said to be sharp (see \cite{Curtosalinas, Sal}) if the multiplication operator  $M_z$ is bounded on $(\hl, K)$ and 
$\ker (M_z^*-\bar{w})=\bigvee \{K(\cdot,w)\eta:\eta\in \mathbb C^n\}$ for all $w\in \D$. 

The statement and the proof in the forward direction of the following lemma is closely related to  \cite[Proposition 2.1]{MisraKoranyihomog}.  
A version of this lemma, without involving the composition by $\varphi$, is  in  \cite{Curtosalinas}.  
\begin{lemma}\label{weakhomlem}
Let $K(z,w): \mathbb D \times \mathbb D \to \mathcal M_n(\mathbb C)$ be a positive definite kernel. Suppose that the multiplication operator $M_z$ on $(\hl, K)$ is bounded. Let $\varphi$ be a fixed but arbitrary function in $\mb$ which is analytic in a neighbourhood of $\;\sigma(M_z)$. If $X$ is a bounded invertible operator on $(\hl, K)$ of the from $M_{g_{\varphi}}C_{\varphi^{-1}}$, where $g_{\varphi}\in \mbox{\rm Hol}(\D, GL_n(\mathbb C))$, then $X$ intertwines $M_z$ and $\varphi(M_z)$,
that is, $M_zX=X\varphi(M_z)$. Moreover, if $K$ is sharp, then the converse of the above statement is also true, that is, if 
 $X$ is a bounded invertible operator on $(\hl, K)$  intertwining $M_z$ and $\varphi(M_z)$, then $X$ is of the form $M_{g_{\varphi}}C_{\varphi^{-1}}$ for some $g_{\varphi}\in \mbox{\rm Hol}(\D, GL_n(\mathbb C))$. 
\end{lemma}
\begin{proof}
If $X$ is a bounded invertible operator  of the form $M_{g_{\varphi}}C_{\varphi^{-1}}$, then an easy computation shows that $X\varphi(M_z)=
M_zX.$

Conversely, assume that $X$ is a bounded invertible operator on $(\hl, K)$ such that $M_z X=X\varphi(M_z)$. Taking adjoint and acting on the vector $K(\cdot,w)\eta$, $w\in \mathbb D, \eta\in \mathbb C^n$, we obtain
\begin{equation}
{\varphi(M_z)}^*X^*K(\cdot,w)\eta=X^*M_z^*K(\cdot,w)\eta=\overbar{w}X^*K(\cdot,w)\eta.
\end{equation}
Thus $X^*K(.,w)\eta\in \ker \big({\varphi(M_z)}^*-\overbar{w}\big)$. It is easy to see that  $\ker \big({\varphi(M_z)}^*-\overbar{w}\big)=
\ker \big(M_z^*-\overbar{\varphi^{-1}(w)}\big)$.
Since $K$ is sharp and the vector $X^*K(\cdot,w)\eta\in \ker \big(M_z^*-\overbar{\varphi^{-1}(w)}\big)$, it follows  that there exists a unique vector $h_{\varphi}(w)\eta\in\mathbb C^n$ such that 
\begin{equation}\label{eqnwkhomlemma1}
X^*K(\cdot,w)\eta=K(\cdot,\varphi^{-1}(w))h_{\varphi}(w)\eta.
\end{equation}
The invertibility  of the matrix  $K(\varphi^{-1}(w),\varphi^{-1}(w))$ ensures the  uniqueness of the vector $h_{\varphi}(w)\eta$.  
It is easily verified that for each $w\in \mathbb D$, the map $\eta\mapsto h_{\varphi}(w)\eta$ defines a linear map on $\mathbb C^n$. Since $X$ is invertible, it follows from \eqref{eqnwkhomlemma1} that $h_{\varphi}(w)$ is invertible. Now for any $w\in \mathbb D$, $\eta\in \mathbb C^n$ and $f\in (\hl, K)$, we see that
\begin{align*}
\left\langle (Xf)(w),\eta \right\rangle  
&=\left\langle Xf, K(\cdot,w)\eta \right\rangle \\
&= \left\langle f, X^* K(\cdot,w)\eta \right\rangle \\
& = \left\langle f, K(\cdot, \varphi^{-1}(w)) h_{\varphi}(w)\eta\right\rangle \\
&=\left\langle (f\circ \varphi^{-1})(w), h_{\varphi}(w) \eta \right\rangle \\
& =\left \langle {h_{\varphi}(w)}^*(f\circ \varphi^{-1})(w),\eta\right \rangle.
\end{align*}

Hence $X=M_{g_{\varphi}}C_{\varphi^{-1}}$,
where $g_{\varphi}(w)={h_{\varphi}(w)}^*$, $w\in \mathbb D$. Since we have already shown that $g_{\varphi}(w)$, $w\in \mathbb D$, is invertible, to complete the proof, we only need to show that the map $w\mapsto g_{\varphi}(w)$ is holomorphic. 

Let $w_0$ be a fixed but arbitrary point in $\D$. Since $K(\varphi^{-1}(w_0),\varphi^{-1}(w_0))$
is invertible, there exists a neighbourhood 
$\Omega_0$ of $w_0$ such that $K(\varphi^{-1}(w_0),\varphi^{-1}(w))$ is invertible for all $w$ in $\Omega_0$. From \eqref{eqnwkhomlemma1}, we have
$$
(X^*K(\cdot,w)\eta)\varphi^{-1}(w_0)=K(\varphi^{-1}(w_0),\varphi^{-1}(w))h_{\varphi}(w)\eta,\;w\in \Omega_0.
$$ 
Therefore
$$h_{\varphi}(w)\eta = K(\varphi^{-1}(w_0),\varphi^{-1}(w))^{-1}(X^*K(\cdot,w)\eta)\varphi^{-1}(w_0),\;w\in \Omega_0.$$
Since the right hand side of the above equality is anti-holomorphic on $\Omega_0$, it follows that the function $h_{\varphi}(w)$ is anti-holomorphic on $\Omega_0$, and therefore $g_{\varphi}$ is holomorphic on $\Omega_0$. Since $w_0$ is arbitrary, we conclude that $g_{\varphi}$ is  holomorphic on $\Omega$. This completes the proof.
\end{proof}

 
\begin{proposition}\label{propweakhomo}
Let $K(z,w):\mathbb D \times \mathbb D \to \mathcal M_n(\mathbb C)$ be a positive definite kernel. Suppose that the multiplication operator $M_z$ on $(\hl, K)$ is bounded. If for each $\varphi\in \mb$, there exists a function $g_{\varphi}\in  {\rm Hol}(\mathbb D, GL_n(\mathbb C))$ such that the operator $M_{g_{\varphi}}C_{\varphi^{-1}}$ on $(\hl, K)$ is bounded and invertible, then $M_z$ on $(\hl, K)$ is weakly homogeneous. Moreover, if $K$ is sharp
and the operator $M_z$ on $(\hl, K)$ is weakly homogeneous, 
then for each $\varphi$ in $\mb$, there exists $g_{\varphi}\in  \mbox{\rm Hol}(\mathbb D, GL_n(\mathbb C))$ such that the weighted composition operator $M_{g_{\varphi}}C_{\varphi^{-1}}$ on $(\hl, K)$ is bounded and invertible .
\end{proposition}
\begin{proof}
%


Let $U$ be a neighbourhood of the identity in $\mbox{\rm M\"ob}$ such that $\varphi(M_z)$ is well-defined for all $\varphi\in U$. By hypothesis, there exists $g_{\varphi}\in 
\mbox{\rm Hol}(\mathbb D, GL_n(\mathbb C))$ such that the operator $M_{g_{\varphi}}C_{\varphi^{-1}}$ on $(\hl, K)$ is bounded and 
invertible. Then by Lemma \ref{weakhomlem}, it follows that the operator $M_z$ satisfies $M_zX=\varphi(M_z)X$, $\varphi\in U$, where $X=M_{g_{\varphi}}C_{\varphi^{-1}}$ on $(\hl, K)$. Hence $M_z$ is similar to $\varphi(M_z)$ for all $\varphi\in U$.
Now a straightforward generalization of \cite[Lemma 2.2]{constantchar} completes the proof in the forward direction. The proof for the second part follows directly from Lemma
\ref{weakhomlem}.
\end{proof}

The theorem appearing below shows that the weak homogeneity of the multiplication operator is preserved under the jet construction.
\begin{theorem}\label{thmweakhom}
Suppose that $K_1$ and $K_2$ are two scalar valued sharp positive definite kernels on $ \mathbb D\times \mathbb D$. 
If the multiplication operators $M_z$ on $(\mathcal H, K_1)$ and $(\hl, K_2)$ are weakly homogeneous, then $M_z$ on   
$\big(\mathcal H,J_n(K_1,K_2)_{|\rm res\, \Delta}\big)$ is also weakly homogeneous.
\end{theorem}

\begin{proof}
Since the operators $M_z$ on $(\mathcal H, K_1)$ and $(\hl, K_2)$ are weakly homogeneous, by Proposition 
\ref{propweakhomo}, for each $\varphi\in \mbox{\rm M\"ob}$, there exist $g_{\varphi}, h_{\varphi}\in \mbox{\rm Hol}(\mathbb D,\mathbb C\setminus \{0\})$ such that the weighted composition operators $M_{g_{\varphi}}C_{\varphi^{-1}}$ on $(\mathcal H, K_1)$ and 
$ M_{h_{\varphi}}C_{\varphi^{-1}} $ on $(\mathcal H, K_2)$ are bounded and  invertible. Then, by the first part of Theorem  \ref{wtcompthm2}, it follows that the operator $ M_{g_{\varphi}(\mathcal J_n h_{\varphi}) (\mathcal B_n\varphi^{-1} )}C_{\varphi^{-1}}$ on $\big(\mathcal H,J_n(K_1,K_2)_{|\rm res\, \Delta}\big)$ is bounded and invertible. An application of Proposition \ref{propweakhomo}, once again,  completes the proof.
\end{proof}
\section{Weakly homogeneous operators in the class 
$\mathcal FB_2(\D)$}\label{secfb_2}
Although there are examples (see \cite[Theorem $(1.1)^\prime$]{wtcomplarge}, \cite[Theorem 3.3]{wtcompsmall}) of scalar valued sharp kernels $K$ such that the composition operators $C_\varphi$, $\varphi \in \mb$, are not bounded on $(\mathcal H, K)$, it does not necessarily follow that the multiplication operator $M_z$ on $(\mathcal H,K)$ fails to be weakly homogeneous. In many other examples excluding the ones in \cite{wtcomplarge} and  \cite{wtcompsmall}, the operator $C_\varphi$ is bounded for all $\varphi$ in $\mb$ showing that the corresponding multiplication operator $M_z$ is weakly homogeneous. 
While the question of the existence of an  operator $M_z$ which is not weakly homogeneous on a Hilbert space $(\mathcal H, K)$, where $K$ is a scalar valued sharp kernel, remains unanswered, in this section, we find such examples where the kernel $K$ takes values in $\mathcal M_2(\mathbb C)$.  

Given a bounded domain $\Omega\subseteq \mathbb C$, a smaller class  $\mathcal FB_n(\Omega) \subseteq B_n(\Omega)$, $n\geq 2$, of operators was  introduced in \cite{Flagstructure}.  A complete set of tractable  unitary invariants and concrete  models were given for operators in this class. For our purposes, it is enough to restrict attention to the case of $\Omega=\mathbb D$ and $n=2$.

\begin{definition}
An operator $T$ on $H_0\bigoplus H_1$ is said to be in $\mathcal FB_2(\D)$ 
if it is of the form 
$ \begin{bmatrix}
T_0 & S\\
0 &  T_1   
\end{bmatrix}$, 
where $T_0,T_1\in B_1(\D)$ and $S$ is a non-zero operator satisfying
$T_0S=ST_1.$ 
\end{definition}

Since $\mathcal FB_2(\D)\subseteq B_2(\D)$, every operator $T$
in $\mathcal FB_2(\D)$ is unitarily equivalent to the adjoint $M^*_z$ of the multiplication operator by the coordinate function $z$ on some reproducing kernel Hilbert space $(\hl, K)$, where $K$ takes values in $\mathcal M_2(\mathbb C)$.
It is known that $\mathcal FB_2(\D)$  contains all homogeneous operators in $B_2(\D)$. In this section, we study weakly homogeneous operators in  $\mathcal FB_2(\D)$. The following proposition is an essential tool in this study.


\begin{proposition}\label{propsimFl}{\rm(\cite[Proposition 3.3]{Flagstructure})}
Let $T$ and $\tilde T$ be any two operators in $\mathcal FB_2(\D).$
If $L$ is a bounded invertible operator which intertwines $T$ and $\tilde T,$ then $L$ and $L^{-1}$ are upper triangular. 
\end{proposition}

\begin{corollary}\label{corFB_2weak}
Let $T=\begin{pmatrix}
T_0&S\\
0&T_1
\end{pmatrix}$ on $\mathcal H_0\oplus \mathcal H_1$
and $\tilde T=\begin{pmatrix}
\tilde T_0&\tilde S\\
0&\tilde T_1
\end{pmatrix}$ on $\tilde{\mathcal H}_0\oplus \tilde{\mathcal H}_1$
be two operators in $\mathcal FB_2(\D).$ Then $T$ is similar to
$\tilde T$ if and only if there exist bounded invertible operators $X:\mathcal H_0\to\tilde {\mathcal H}_0$, $Y:\mathcal H_1\to\tilde {\mathcal H}_1$ and a bounded operator 
$Z:\mathcal H_1\to\tilde {\mathcal H}_0$ such that 
\begin{itemize}
\item[\rm (i)]$X T_0=\tilde T_0 X,\;\;Y T_1=\tilde{T}_1 Y$,
\item[\rm (ii)]$X S+Z T_1=\tilde{T}_0 Z+\tilde S Y.$
\end{itemize}
\end{corollary}
\begin{proof}
Suppose that $T$ is similar to $\tilde T$. Let $A=\left(\begin{smallmatrix}
X&Z\\
W&Y\\
\end{smallmatrix}\right): 
{\mathcal H}_0\oplus {\mathcal H}_1\to
\tilde{\mathcal H}_0\oplus \tilde{\mathcal H}_1$ be an invertible operator such that $AT=\tilde T A.$ By Proposition $\ref{propsimFl}$,  $W=0.$ Also, the intertwining relation is equivalent to  
$$
\begin{pmatrix}
XT_0&XS+ZT_1\\
0&YT_1
\end{pmatrix}=
\begin{pmatrix}
\tilde{T}_0X&\tilde{T}_0 Z+\tilde S Y\\
0&\tilde{T}_1 Y
\end{pmatrix}.
$$
Applying Proposition $\ref{propsimFl}$ once again, we see that $A^{-1}$ is also upper triangular. Hence using Lemma $\ref{lemwtcomp}$, we conclude that $X$ and $Y$ are invertible.

Conversely, assume that there exist bounded 
invertible operators $X:\mathcal H_0\to\tilde {\mathcal H}_0$, $Y:\mathcal H_1\to\tilde {\mathcal H}_1$ and a linear operator 
$Z:\mathcal H_1\to\tilde {\mathcal H}_0$ satisfying $\rm (i)$ and $\rm (ii)$ of this Corollary. Let $A$ be the operator $\left(\begin{smallmatrix}
X&Z\\
0&Y
\end{smallmatrix}\right).$ 
It is easily verified that $A$ is invertible
and the intertwining relation $AT=\tilde T A$ holds.
\end{proof}
The following lemma appeared in the PhD thesis of Dayal Dash Purohit \cite{DDPthesis}, see also \cite{misra-reza}.
\begin{lemma}\label{lemautB_1D}
Let $T\in B(\mathcal H)$ be an operator in $B_1(\mathbb D)$ with $\sigma(T)
=\bar{\mathbb D}.$ Then the operator $\varphi(T)$ belongs to $B_1(\mathbb D)$ for all $\varphi$ in \rm M\"ob.
\end{lemma}
\begin{lemma}\label{lemautFB_2D}
Let $T=\begin{pmatrix}
T_0&S\\
0&T_1
\end{pmatrix}$ be an operator in $\mathcal FB_2(\mathbb D)$
with $\sigma(T)=\sigma(T_0)=\sigma(T_1)=\bar{\mathbb D}$. Then the operator $\varphi(T)$ belongs to
$\mathcal FB_2(\mathbb D)$ for all $\varphi$ in $\mb$.
\end{lemma}

\begin{proof}
A routine verification, using the intertwining relation $T_0S=ST_1$, shows that
$$\varphi(T)=
\begin{psmallmatrix}
\varphi(T_0)&\varphi'(T_0)S\\
0& \varphi(T_1)
\end{psmallmatrix},\,\, \mbox{and}\,\,\varphi(T_0)\varphi'(T_0)S=\varphi'(T_0)S\varphi(T_1).$$
Also, by Lemma $\ref{lemautB_1D}$, the operators $\varphi(T_0)$ and $\varphi(T_1)$ belong to $B_1(\mathbb D)$. Hence the operator 
$\varphi(T)$ belongs to $\mathcal FB_2(\mathbb D)$. 
\end{proof}
%

The following corollary
is an immediate consequence of Corollary \ref{corFB_2weak} and Lemma \ref{lemautFB_2D}. 
\begin{corollary}\label{corweakhom}
Let $T=\begin{pmatrix}
T_0&S\\
0&T_1
\end{pmatrix}$ be an operator in $\mathcal FB_2(\mathbb D)$ with $\sigma(T)=\sigma(T_0)=\sigma(T_1)=\bar{\mathbb D}$. Then $T$ is weakly homogeneous if and only if for each $\varphi $ in $\mb$, there exist bounded invertible operators $X_{\varphi}:\mathcal H_0\to \mathcal H_0$, $Y_{\varphi}:\mathcal H_1\to \mathcal H_1$ and a bounded operator $Z_{\varphi}:\mathcal H_1 \to \mathcal H_0$ such that the followings hold: 
\begin{itemize}
\item[\rm {(i)}]$X_{\varphi} T_0=\varphi(T_0)X_{\varphi},~Y_{\varphi} T_1=\varphi(T_1)Y_{\varphi},$
\item[(ii)]$X_{\varphi}S+Z_{\varphi}T_1=\varphi(T_0)Z_{\varphi}+
\varphi'(T_0)S Y_{\varphi}.$
\end{itemize}
\end{corollary}
\subsection{A useful Lemma}
%
Let $K_1, K_2:\mathbb D\times\mathbb D \to \mathbb C$ be two  positive definite kernels. Let $M^{(1)}$ and $M^{(2)}$ denote the operators of multiplication by the coordinate function $z$ on $(\hl, K_1)$ and $(\hl, K_2)$, respectively. Assume that $M^{(1)}$ and $M^{(2)}$ are bounded.  
The following lemma is a key ingredient in constructing operators in $\mathcal FB_2(\mathbb D)$ that are not weakly homogeneous.

\begin{lemma}\label{lemrosenbl}
Let $\varphi$ be a fixed but arbitrary function in $\mb$ which is analytic in a neighbourhood of $\sigma(M^{(1)})$, and let $\psi$ be a function in ${\rm Hol}(\mathbb D)$ such that the weighted composition operator $M_{\psi}C_{\varphi^{-1}}$ is bounded from $(\hl, K_1)$ to $(\hl, K_2)$.
If $X$ is a bounded linear operator from $(\mathcal H,K_1)$ 
to $(\mathcal H,K_2)$ of the form  
$X(f)=\psi (\varphi^{-1})'(f'\circ\varphi^{-1})+
\chi(f\circ\varphi^{-1})$, $f\in (\hl, K_1)$ for some $\chi\in {\rm Hol}
(\mathbb D)$, then $X$ satisfies  
\begin{equation}\label{eqndiff}
X\varphi(M^{(1)})-M^{(2)} X = M_{\psi} C_{\varphi^{-1}}.
\end{equation}
Moreover, if $K_1$ is sharp, then the converse of the above statement is also true, that is, if $X:(\hl, K_1)\to (\hl, K_2)$ is a bounded linear operator satisfying \eqref{eqndiff}, then $X$ is of the form 
$X(f)= \psi (\varphi^{-1})'(f'\circ\varphi^{-1})+
\chi (f\circ \varphi^{-1})$, $f\in (\hl, K_1)$,
for some $\chi\in\mbox{\rm Hol}(\mathbb D)$.
%
\end{lemma}
\noindent(Here $\psi (\varphi^{-1})^\prime$ denotes the pointwise product of the two functions $\psi$ and $(\varphi^{-1})^\prime$. Similarly, $\psi (\varphi^{-1})^\prime (f'\circ\varphi^{-1})$ denotes the pointwise product of $\psi (\varphi^{-1})^\prime$ and $(f'\circ\varphi^{-1})$. Finally, $\chi(f\circ \varphi^{-1})$ is the pointwise product of $\chi$ and $f\circ \varphi^{-1}$. This convention is adopted throughout this paper.)
\begin{proof}
Suppose that $X$ is bounded linear operator taking $f$ to $\psi (\varphi^{-1})'(f'\circ\varphi^{-1})+
\chi(f\circ\varphi^{-1})$, $f\in (\hl, K_1)$.
Then we see that
\begin{align*}
(X\varphi(M^{(1)})-M^{(2)} X)f=& X(\varphi f)-z Xf\\
= & \psi (\varphi^{-1})'((\varphi f)'\circ\varphi^{-1})+\chi((\varphi f)\circ 
\varphi^{-1})-z\psi (\varphi^{-1})'(f'\circ\varphi^{-1})-z\chi(f\circ 
\varphi^{-1})\\
= & \psi(\varphi^{-1})'\big( (\varphi'\circ \varphi^{-1})(f\circ\varphi^{-1})
+(\varphi\circ\varphi^{-1}) (f'\circ \varphi^{-1})\big)
+z\chi(f\circ\varphi^{-1})\\
&\quad\quad\quad\quad\quad\quad\quad\quad\quad\quad\quad\quad\quad\quad\quad\quad-z\psi (\varphi^{-1})'(f'\circ\varphi^{-1})-z\chi (f\circ\varphi^{-1})\\
= & \psi(\varphi^{-1})'(\varphi'\circ \varphi^{-1})(f\circ\varphi^{-1})+ \psi(\varphi^{-1})'z(f'\circ \varphi^{-1})-z\psi (\varphi^{-1})'(f'\circ\varphi^{-1})\\
= & \psi(f\circ \varphi^{-1}). 
\end{align*} 
Here for the last equality we have used the identity $(\varphi^{-1})'(\varphi'\circ \varphi^{-1})=1$.

For the converse, assume that $K_1$ is sharp and $X:(\hl, K_1) \to (\hl,K_2)$ is a bounded linear operator
satisfying $\eqref{eqndiff}.$ Then taking adjoint and acting 
on $K_2(\cdot,z)$, $z\in \mathbb D$, 
we obtain
\begin{align}\label{eqndiff1}
\begin{split}
\varphi(M^{(1)})^*X^*K_2(\cdot,z)-\overbar{z}X^*K_2(\cdot,z)
&=(M_{\psi} C_{\varphi^{-1}})^*K_2(\cdot,z)\\
&=\overbar{\psi(z)}K_1(\cdot,\varphi^{-1}z).
\end{split}
\end{align}  
Here the last equality follows from exactly the same argument as in \eqref{eqnwtcomkernelfns}.
Furthermore, since $\big(\varphi(M^{(1)})^*-\overbar{\varphi(w)}\big)K_1(\cdot,w)=0$, 
$w\in \mathbb D$, differentiating with respect to $\overbar{w}$, we see  that 
\begin{equation}
\big(\varphi(M^{(1)})^*-\overbar{\varphi(w)}\big)\bar\partial K_1(\cdot,w)
=\overbar{\varphi'(w)}K_1(\cdot,w),\;\;w \in \mathbb D.
\end{equation}
Replacing $w$ by $\varphi^{-1}z$ 
in the above equation and combining it with 
$\eqref{eqndiff1}$, we see that 
\begin{equation}
\big(\varphi(M^{(1)})^*-\overbar{z}\big)X^*K_2(\cdot,z)
=\big(\varphi(M^{(1)})^*-\overbar{z}\big)
\left(\frac{\overbar{\psi(z)}}{\overbar{\varphi'(\varphi^{-1}z)}}
\bar\partial K_1(\cdot,\varphi^{-1}z)\right).
\end{equation}
Consequently, the vector $X^*K_2(\cdot,z)-\tfrac{\overbar{\psi(z)}}{\overbar{\varphi'(\varphi^{-1}z)}}
\bar\partial K_1(\cdot,\varphi^{-1}z)\in \ker \big(\varphi(M^{(1)})^*-\overbar{z}\big)$.
Since $K_1$ is sharp, we have that $\ker \big(\varphi(M^{(1)})^*-\bar{z}\big)=
\bigvee \{K_1(\cdot,\varphi^{-1}z)\}$ (see the proof of Lemma \ref{weakhomlem}). Therefore 
$$X^*K_2(\cdot,z)-\frac{\overbar{\psi(z)}}{\overbar{\varphi'
(\varphi^{-1}z)}}\bar\partial K_1
(\cdot,\varphi^{-1}z)= \overbar{\chi(z)} K_1(\cdot,\varphi^{-1}z),$$
for some $\chi\in \mbox{Hol}(\mathbb D)$ (the holomorphicity of $\chi$ can be proved by a similar argument used at the end of Lemma \ref{weakhomlem}).

Finally, for $f\in (\mathcal H, K_1)$ and $z\in \mathbb D$, we see that
\begin{align*}
\big(Xf\big)(z)& =\left\langle Xf, K_2(\cdot,z)\right\rangle\\
&=\left\langle f, X^*K_2(\cdot,z)\right\rangle\\
&=\left\langle f, \frac{\overbar{\psi(z)}}
{\overbar{\varphi'(\varphi^{-1}z)}}
\bar\partial K_1(\cdot,\varphi^{-1}z)+\overbar{\chi(z)} K_1(\cdot,\varphi^{-1}z)\right\rangle\\
& =\psi(z) (\varphi^{-1})'(z)(f'\circ\varphi^{-1})(z)
+\chi(z)(f\circ \varphi^{-1})(z),
\end{align*}
where the equality $\left\langle f, \bar{\partial} K_1(\cdot,\varphi^{-1}z)\right\rangle = (f'\circ\varphi^{-1})(z)$ follows from \cite[Lemma 4.1]{Curtosalinas}.
This completes the proof.
\end{proof}

\begin{notation}
Recall that $\mathcal H^{(\lambda)}$, $\lambda>0$, denote the Hilbert space determined by the positive definite kernel $K^{(\lambda)}$, where $K^{(\lambda)}(z,w):=\frac{1}{(1-z\bar{w})^{\lambda}}$, $z,w\in \mathbb D$. 
Note that 
\begin{equation}\label
{eqnK^lambda}
K^{(\lambda)}(z,w)=\sum_{n=0}^\infty \frac{(\lambda)_n}{n!} (z\bar{w})^n,\;\;z,w\in \D,
\end{equation} 
where $(\lambda)_n$ is the Pochhammer symbol given by $\frac{\Gamma(\lambda+n)}{\Gamma(\lambda)}$.

For any $\gamma\in \mathbb R$, let $K_{(\gamma)}$ be the positive definite kernel given by 
\begin{equation}\label{eqnK_nu}
K_{(\gamma)}(z,w):=\sum_{n=0}^\infty (n+1)^{\gamma} (z\bar{w})^n, z, w \in \mathbb D.
\end{equation}
Note that $K_{(0)}=K^{(1)}$ (the Szeg\"o kernel of the unit disc $\D$) and $K_{(1)}=K^{(2)}$(the Bergman kernel of the unit disc $\D$). 
The kernel $K_{(-1)}$ is known as the Dirichlet kernel of the unit disc $\D$ and the Hilbert space $(\hl, K_{(-1)})$ is known as the Dirichlet space. 

For two sequences $\{a_n\}$ and $\{b_n\}$ of positive real numbers, we write $a_n \sim b_n$ if there exist constants $c_1, c_2>0$
such that $c_1 b_n\leq a_n\leq c_2 b_n$ for all $n\in \z$.  From \eqref{eqnK^lambda} and \eqref{eqnK_nu}, it is clear that $\|z^n\|^2_{\mathcal H^{(\lambda)}}=\frac{n!}{(\lambda)_n}=
\frac{\Gamma(n+1)\Gamma(\lambda)}{\Gamma(n+\lambda)}$ and $\|z^n\|^2_{(\hl, K_{(\gamma)})}=(n+1)^{-\gamma}$, $n\in \z$. 
 Using the identity $\lim_{n\to \infty}\frac{\Gamma(n+a)}{\Gamma(n)n^a}=1$, $a\in \mathbb C$, we see that
 \begin{equation}\label{eqnlambdnu}
\|z^n\|^2_{\mathcal H^{(\lambda)}} \sim \|z^n\|^2_{(\hl, K_{(\lambda-1)})}~~ ~\mbox{for all}~~\lambda>0.
\end{equation}
\end{notation}

Recall that a Hilbert space $\mathcal H$ consisting of holomorphic functions on the unit disc $\D$ is said to be M\"obius invariant if for each $\varphi\in\mb$, $f\circ\varphi \in \mathcal H$ whenever $f\in \mathcal H$. By an application of the closed graph Theorem, it follows that $\hl$ is M\"obius invariant if and only if the composition operator $C_{\varphi}$ is bounded on $\hl$ for each $\varphi\in \mb$. If the 
multiplication operator $M_z$ is bounded on some M\"obius invariant Hilbert space $\hl$, then by Proposition \ref{propweakhomo}, it follows that $M_z$ is weakly homogeneous on $\hl$.
It is known that the Hilbert spaces $\hl^{(\lambda)}$, $\lambda>0$, and
$(\hl, K_{(\gamma)})$, $\gamma\in \mathbb R$, are M\"obius invariant (see \cite{Zorboska}, \cite{cowen-maccluer}). 
%
We record this fact as a lemma.
\begin{lemma}\label{lemmobinv}
The Hilbert spaces $\hll$, $\lambda>0$, and $(\hl, K_{(\gamma)})$, $\gamma\in \mathbb R$, are M\"obius invariant. Consequently, 
the composition operator $C_{\varphi}$ is bounded and invertible on $\hll$, $\lambda>0$, as well as on $(\hl, K_{(\gamma)})$, $\gamma\in \mathbb R,$ for all  $\varphi\in \mb$.
\end{lemma}

\begin{corollary}\label{cormobinv}
For any $\gamma\in \mathbb R$, the operator $M_z^*$ on $(\hl, K_{(\gamma)})$ is a weakly homogeneous operator in $B_1(\D)$. 
Moreover, it is similar to a homogeneous operator if and only if  $\gamma > -1$. In particular, $M_z^*$
on the Dirichlet space is a weakly homogeneous operator which is not similar to any homogeneous operator.
\end{corollary}
\begin{proof}
Note that $M_z$ on $(\hl, K_{(\gamma)})$ is unitarily equivalent to the weighted shift with the weight sequence $\{w_n\}_{n\in \z}$, where $w_n=\Big(\frac{n+1}{n+2}\Big)^{\frac{\gamma}{2}}$, $n\in\z$. Since $\sup_{n\in \z} w_n<\infty$, it follows that $M_z$ on $(\hl, K_{(\gamma)})$ is bounded. Thus by Lemma \ref{lemmobinv} and 
Proposition \ref{propweakhomo}, $M_z$ on $(\hl, K_{(\gamma)})$ is weakly homogeneous . 

Recall that for an operator $T$, $r_1(T)$ is defined as $\lim_{n\to \infty} \big(m(T^n)\big)^{\frac{1}{n}}$ (which always exists,  
see \cite{Shields}), where $m(T)=\inf \big\{\|Tf\|:\|f\|=1\big\}$. 
For the multiplication operator $M_z$ on $(\hl, K_{(\gamma)})$, it is easily verified that $r_1(M_z)=r(M_z)=1$, where $r(M_z)$ is the spectral radius of $M_z$. Hence, by a theorem of Seddighi (cf. \cite{Seddighi}), we conclude that $M_z^*$ on $(\hl, K_{(\gamma)})$ belongs to $B_1(\D)$.

Finally, assume that $M_z^*$ on $(\hl, K_{(\gamma)})$ is similar to a homogeneous operator, say $S$. Since $B_1(\D)$ is closed under similarity, the  operator $S$ belongs to $B_1(\D)$. Also, since upto unitary equivalence, every homogeneous operator in $B_1(\D)$ is of the form $M_z^*$ on $(\hl, K^{(\lambda)})$, $\lambda>0$, it follows that   $M_z^*$ on $(\hl, K_{(\gamma)})$ is similar to $M_z^*$ on $\hll$ for some $\lambda>0$, and therefore, by  \cite[Theorem $2^\prime$]{Shields}, we have $\|z^n\|^2_{(\hl, K_{(\gamma)})}\sim \|z^n\|^2_{\hll}$. Then, by \eqref{eqnlambdnu}, 
$\|z^n\|^2_{(\hl, K_{(\gamma)})}\sim \|z^n\|^2_{(\hl, K_{(\lambda-1)})}$. Hence $\gamma=\lambda-1$. Since $\lambda>0$, it follows that $\gamma>-1$.
For the converse, let $\gamma>-1$. Again using \cite[Theorem $2^\prime$]{Shields} and \eqref{eqnlambdnu}, it follows that $M_z^*$ on $(\hl, K_{(\gamma)})$
is similar to the homogeneous operator $M_z^*$ on  $\hl^{(\gamma+1)}$.
\end{proof}

The lemma given below shows that the linear map $f\mapsto f^\prime$ is bounded from $\hll$ to $\hl^{(\lambda+2)}$.
\begin{lemma}\label{lemdiffbdd}
Let $\lambda>0$, and let $f\in \rm Hol(\D)$. 
Then $f\in \hll$ 
if and only if $f^\prime \in \hl^{(\lambda+2)}$. Moreover, $\|f^\prime\|_{\hl^{(\lambda+2)}}\leq \sqrt{\lambda(\lambda+1)}\|f\|_{\hll}$, $f\in \hll$. Consequently, the differential operator $D$ that maps $f$ to $f^\prime$ is bounded from 
$\hll$ to $\hl^{(\lambda+2)}$ with $\|D\|\leq \sqrt{\lambda(\lambda+1)}$. 
\end{lemma}
\begin{proof}
Let $\sum_{n=0}^\infty \alpha_n z^n$, $z\in \D$, be the power series representation of $f$. Then we have $f^\prime (z)=\sum_{n=0}^\infty (n+1)\alpha_{n+1}z^n$, $z\in \D$.
Note that $\|z^n\|^2_{\mathcal H^{(\lambda)}}=\frac{n!}{\lambda(\lambda+1)\cdots(\lambda+n-1)}$ for all $n\geq 0$. By a straightforward computation, we see that 
\begin{align*}
\lambda\|z^{n+1}\|^2_{\hll}\leq (n+1)^2\|z^n\|^2_{\hl^{(\lambda+2)}}& 
\leq \lambda(\lambda+1)\|z^{n+1}\|^2_{\hll},\;\;n\geq 0.
\end{align*}
Consequently, we have 
\begin{align}\label{eqndifhll}
\begin{split}
\lambda\sum_{n=0}^\infty |{\alpha}_{n+1}|^2\|z^{n+1}\|_{\hl^{(\lambda)}}^2& \leq  \sum_{n=0}^\infty |{\alpha}_{n+1}|^2(n+1)^2\|z^n\|_{\hl^{(\lambda+2)}}^2\\ 
&\leq \lambda(\lambda+1) \sum_{n=0}^\infty|\alpha_{n+1}|^2\|z^{n+1}\|^2_{\hll}.
\end{split}
\end{align}
Therefore,  
$\sum_{n=0}^\infty|\alpha_n|^2\|z^n\|^2_{\hll} <\infty$ if and only if 
$\sum_{n=0}^\infty |{\alpha}_{n+1}|^2(n+1)^2\|z^n\|_{\hl^{(\lambda+2)}}^2 < \infty$. Hence $f\in \hl^{(\lambda)}$ if and only if $f^\prime\in \hl^{(\lambda+2)}$.
The rest of the proof follows from \eqref{eqndifhll}.
\end{proof}
The proof of the corollary given below follows from  Lemma \ref{lemdiffbdd} together with the fact that the inclusion operator $f\mapsto f$ is bounded from $\hl^{(\lambda+2)}$ to $\hlm$ whenever $\mu-\lambda\geq 2$.
\begin{corollary}\label{corderrl}
Let $\lambda, \mu$ be two positive real numbers such that $\mu-\lambda\geq 2$. Then 
the linear map $f\mapsto f^{\prime}$ is bounded from $\hll$ to $\hlm$.
\end{corollary}
\begin{lemma}\label{lemstrongir}
Let $\lambda, \mu$ be two positive real numbers such that $\mu-\lambda<2$,  
and let $\psi$, $\chi \in {\rm Hol}(\D)$.
Let $X$ be the linear map given by $X(f)=\psi f'+\chi f$, $f\in {\rm Hol}(\D)$. If $X$ is bounded from $\mathcal H^{(\lambda)}$ to $\mathcal H^{(\mu)}$, then $\psi$ is identically zero.
\end{lemma}

\begin{proof}
Let $\psi(z)=\sum_{j=0}^\infty {\alpha}_j z^j$ and 
$\chi(z)=\sum_{j=0}^\infty {\beta}_j z^j$ be the power series representations of $\psi$ and $\chi$, respectively. Then for $n\geq 1$, we see that
\begin{align*}
\|X(z^n)\|^2_{\hlm} &=\|nz^{n-1}\psi(z)+z^n \chi(z)\|^2_{\hlm}\\
& =\|nz^{n-1}{\alpha}_0+\sum_{j=1}
^\infty (n {\alpha}_j+{\beta}_{j-1})z^{j+n-1}\|^2_{\hlm}\\
&=|{\alpha}_0|^2 n^2\|z^{n-1}\|^2_{\hlm}+\sum_{j=1}
^\infty |n {\alpha}_j+{\beta}_{j-1}|^2 \|z^{j+n-1}\|^2_{\hlm}.
\end{align*}
Since $X$ is bounded from $\hll$ to $\hlm$, we have that $\|X(z^n)\|^2_{\hlm}
\leq \|X\|^2 \|z^n\|^2_{\hll}.$ Consequently, for 
$n\geq 1$, 
\begin{equation}\label{eqnstrongir}
|{\alpha}_0|^2 n^2\|z^{n-1}\|^2_{\hlm}+\sum_{j=1}
^\infty |n {\alpha}_j+{\beta}_{j-1}|^2 \|z^{j+n-1}\|^2_{\hlm} \leq \|X\|^2 \|z^n\|^2_{\hll}.
\end{equation}

From \eqref{eqnlambdnu}, we have
$\|z^n\|^2_{\mathcal H^{(\lambda)}}
\sim n^{-(\lambda-1)}$ 
and $\|z^n\|^2_{\mathcal H^{(\mu})}\sim n^{-(\mu-1)}$. Thus, by \eqref{eqnstrongir}, there exists a constant $c>0$ such that  
$|{\alpha}_0|^2n^{-(\mu-3)} \leq c  n^{-(\lambda-1)}$. Equivalently, $|\alpha_0|^2 \leq c n^{\mu-\lambda-2}.$
Since $\mu-\lambda-2<0,$ taking limit as $n\to \infty $, we obtain 
${\alpha}_0=0.$

For $j\geq 1$, using $\|z^{j+n-1}\|^2_{\hlm}\sim (j+n-1)^{-(\mu-1)}
\sim n^{-(\mu-1)}$ in \eqref{eqnstrongir},
we see that 
$$\left|{\alpha}_j+\frac{\beta_{j-1}}{n}\right|^2\leq d n^{(\mu-1)-(\lambda-1)-2}=
d n^{(\mu-\lambda-2)}$$
for some constant $d>0.$ As before, since $\mu-\lambda-2<0$, 
taking $n\to \infty $, we obtain 
${\alpha}_j=0$ for $j\geq 1$. Hence $\psi$ is identically zero, completing the proof of the lemma.
\end{proof}

From Lemma \ref{lemstrongir},  the converse to the statement in Corollary  \ref{corderrl} follows and consequently, we have the following proposition.
\begin{proposition}\label{cordiffiff} 
The linear map $f\mapsto f^\prime$, $f\in {\rm Hol}(\D)$, is bounded from $\hll$ to $\hlm$ if and only if $\mu-\lambda\geq 2$.
\end{proposition}

Recall that for two Hilbert spaces $\mathcal H_1$ and $\mathcal H_2$ consisting of holomorphic functions on the unit disc $\D$, the multiplier algebra ${\rm Mult}(\hl_1,\hl_2)$ is defined as
$${\rm Mult} (\hl_1, \hl_2):=\left\{ \psi\in \mbox{\rm Hol}(\D):\psi f\in \hl_2 \mbox{~whenever~} f\in \hl_1\right\}.$$ When 
$\mathcal H_1=\mathcal H_2$, we write ${\rm Mult}(\hl_1)$ instead of ${\rm Mult} (\hl_1, \hl_1)$.
By the closed graph theorem, it is easy to see that $\psi \in {\rm Mult}(\hl_1, \hl_2)$ if and only if the multiplication operator $M_{\psi}$ is bounded from $\hl_1$ to $\hl_2$. 

For $\mu\geq \lambda>0$, since $\hll\subseteq \hlm$, it follows that $\psi f \in \hlm$ whenever $f\in \hll$ and $\psi\in {\rm Mult}(\hll)$. Hence
\begin{equation}\label{eqnmult_}
{\rm Mult}(\hll)\subseteq {\rm Mult}(\hll, \hlm),\;\;0<\lambda\leq \mu.
\end{equation}
It is known that for $\lambda\geq 1$, ${\rm Mult}(\hll)= H^\infty (\mathbb D)$, 
where $ H^\infty (\mathbb D)$ is the algebra of all bounded holomorphic functions on the unit disc $\D$. Thus, from \eqref{eqnmult_}, we conclude that 
\begin{equation}\label{eqnmultlm}
 H^\infty (\mathbb D)\subseteq  {\rm Mult} (\hll,\hlm),\;\; 1\leq \lambda\leq \mu.
\end{equation}
On the other hand, if $\lambda > \mu$, then ${\rm Mult}(\hll,\hlm) = \{0\}$, and hence we make the assumption $\lambda \leq \mu$ without loss of generality. 

The proposition given below describes a class a weakly homogeneous operators in $\mathcal FB_2(\mathbb D)$.
\begin{proposition}\label{thmweakhominFB_2}
Let $0< \lambda\leq \mu$ and $\psi \in {\rm Mult}(\hll, \hlm)$. Let $T=\begin{pmatrix}
M_z^* & M_\psi^*\\
0 & M_z^*
\end{pmatrix}
$ on $\hll\oplus \hlm$. If $M_{\psi}$ is bounded and invertible on $\hll$ as well as on $\hlm$, then $T$ is weakly homogeneous. 
\end{proposition}
\begin{proof}
It suffices to show that $T^*$ is weakly homogeneous.
By a routine computation, we obtain 
$\varphi(T^*)= \begin{pmatrix}
M_{\varphi}& 0\\
M_{\psi\varphi'} & M_{\varphi}
\end{pmatrix} \mbox{~~on}~~ \hll\oplus\hlm.$
By Lemma \ref{lemmobinv}, the operator $C_{\varphi^{-1}}$, $\varphi\in \mb$, is bounded and invertible on $\hll$ as well as on  $\hlm$. Also, by hypothesis, $M_{\psi}$ is bounded and invertible on $\hll$ as well as on $\hlm$. 
Thus $M_{\psi}C_{\varphi^{-1}}$ is bounded and invertible on $\hlm$. 
For $\varphi\in \mb$, set $$L_{\varphi}:=\begin{pmatrix}
M_{(\psi\circ \varphi^{-1})(\varphi'\circ \varphi^{-1})}
C_{\varphi^{-1}}& 0\\
0& M_{\psi}C_{\varphi^{-1}} 
\end{pmatrix} \;\;\mbox{on}\; \hll\oplus\hlm.$$
Using the equality $M_{\psi\circ\varphi^{-1}}
=\Cphn M_{\psi}C_{\vphi}$, we see that the operator $M_{\psi\circ\varphi^{-1}}$ is bounded and invertible on $\hll$. Consequently,  $M_{\psi\circ\varphi^{-1}} C_{\varphi^{-1}}$ is bounded and invertible on $\hll$. 
Therefore, to prove that $L_{\varphi}$ is 
bounded and invertible, it suffices to show that  
the operator 
$M_{(\varphi'\circ \varphi^{-1})}$
is bounded and invertible on $\hll.$ 
Take $\vphi$ to be  
$\varphi_{\theta, a}^{-1}$ and note that 
\begin{equation}\label{eqnderphi}
\big((\varphi_{\theta, a}^{-1})'\circ \varphi_{\theta, a}\big)(z)
= \frac{1}{(\varphi_{\theta, a})'(z)}= e^{-i\theta} \frac{(1-\bar{a}z)^2}
{(1-|a|^2)},\;\; z\in \D,
\end{equation}
which is a polynomial. Hence $M_{((\varphi_{\theta, a}^{-1})'\circ  \varphi_{\theta, a})}$ 
is bounded on $\hll.$ Also, from the above equality, we see that $M_{((\varphi_{\theta, a}^{-1})'\circ  \varphi_{\theta, a})}$ is invertible on $\hll$ if and only if 
$(1-\bar{a}z)^{-2}$ is in ${\rm Mult}(\hll)$. To verify this, let $f\in \hll.$
Note that 
\begin{equation}\label{eqnL_fi}
\Big(\frac{1}{(1-\bar{a}z)^2} f\Big)'
=\frac{1}{(1-\bar{a}z)^2} f'+ \frac{2\bar{a}}
{(1-\bar{a}z)^3} f.
\end{equation}
Since $f\in \hll$ and $\hll\subseteq \hl^{(\lambda+2)}$, we have $f\in \hl^{(\lambda+2)}$. Also by Lemma \ref{lemdiffbdd}, $f^\prime\in \hl^{(\lambda+2)}$.
Since the functions $\frac{1}{(1-\bar{a}z)^2}$ 
and $\frac{2\bar{a}}{(1-\bar{a}z)^3}$ belong to 
$ H^{\infty}(\mathbb D)$ and 
${\rm Mult}(\hl^{(\lambda +2)})
= H^{\infty}(\mathbb D),$ 
it follows that both of the functions $\frac{1}{(1-\bar{a}z)^2} f'$ and $\frac{2\bar{a}}{(1-\bar{a}z)^3} f$ belong to $\hl^{(\lambda+2)}$. Thus, by \eqref{eqnL_fi},
$\big(\frac{1}{(1-\bar{a}z)^2} f\big)'$ 
belongs to $\mathcal H^{(\lambda +2)}.$ Hence, again applying Lemma \ref{lemdiffbdd}, we conclude that $\frac{1}{(1-\bar{a}z)^2} f$ belongs to $\hll.$ 
Finally, a straightforward calculation shows that 
\begin{equation}
T^*L_{\varphi}=L_{\varphi}\varphi(T^*)=
\begin{pmatrix}
M_{z(\psi\circ\varphi^{-1})(\varphi^\prime\circ \varphi^{-1})}C_{\varphi^{-1}}& 0\\
M_{\psi(\psi\circ\varphi^{-1})(\varphi^\prime\circ \varphi^{-1})}C_{\varphi^{-1}}& M_{z\psi}C_{\varphi^{-1}}
\end{pmatrix},
\end{equation}
completing the proof. 
\end{proof} 
%

Let $C(\bar{\mathbb D})$ denote the space of all continuous functions on $\bar{D}$. If $\psi$
is an arbitrary function in $C(\bar{\mathbb D})\cap \rm Hol(\mathbb D)$, then it is easy to see that $\psi\in H^{\infty}(\D)$. Furthermore, if 
$1\leq \lambda\leq \mu$, then by 
\eqref{eqnmultlm}, we see that $\psi\in {\rm Mult~}(\hll,\hlm)$.

The theorem given below gives several examples of operators in the class $\mathcal FB_2(\D)$ that are  weakly homogeneous and the ones that are not. 
\begin{theorem}\label{thmwkhm0}
Let $1\leq\lambda\leq \mu<\lambda+2$, and let $\psi$ be a non-zero function in $C(\bar{\mathbb D})\cap \rm Hol(\mathbb D)$. The operator $T=\begin{pmatrix}
M_z^* & M_\psi^*\\
0 & M_z^*
\end{pmatrix}
$ 
on $\hll\oplus\hlm$ is weakly homogeneous if and only if $\psi$
is non-vanishing on $\bar{\mathbb D}$.
\end{theorem}
\begin{proof}
Suppose that $\psi$ is non-vanishing on $\bar{\D}$. Since $\psi$ is continuous on $\bar{\mathbb D}$, $\psi$ must be bounded below. Therefore $\textstyle{\tfrac{1}{\psi}}$ is a bounded analytic function on $\mathbb D$. Further, since $\lambda, \mu\geq 1$, we have  ${\rm Mult}(\hll)={\rm Mult}(\hlm)=H^{\infty}(\mathbb D)$. Hence the operator $M_{\scriptstyle {\tfrac{1}{\scriptstyle\psi}}}$ is bounded on $\hll$ as well as on $\hlm$. Consequently, the operator $M_{\psi}$ is bounded and invertible on  $\hll$ as well as on $\hlm$. Hence, by
Proposition \ref{thmweakhominFB_2}, $T$ is weakly homogeneous.

Conversely, assume that $T$ is weakly homogeneous. 
It is easily verified that $T\in \mathcal FB_2(\D)$
and $T$ satisfies the hypothesis of Corollary \ref{corweakhom}. Therefore,
for each $\varphi $ in 
$\mbox{\rm M\"ob}$, there exists 
bounded operators 
$X_{\varphi} :\hll\to \hll$, $Y_{\varphi}:\hlm\to \hlm$ and  $Z_{\varphi}:\hlm \to \hll$,
with $X_{\varphi}, Y_{\varphi}$ invertible, such that the following holds:
\begin{align}\label{eqnweakhomFB2}
\begin{split}
& \quad X_{\varphi} T_0=\varphi(T_0)X_{\varphi},~Y_{\varphi} T_1=\varphi(T_1)Y_{\varphi},\\
& X_{\varphi} M_{\psi}^*+Z_{\varphi}T_1=\varphi(T_0)Z_{\varphi}+
M_{\psi}^*\varphi^\prime(T_1)Y_{\varphi},
\end{split}
\end{align}
where $T_0$ is $M_z^*$ on $\hll$ and 
$T_1$ is $M_z^*$ on $\hlm$. Note that 
${\varphi(T_0)}^*=\hat{\varphi}(T_0^*)$, where
$\hat{\varphi}(z):=\overbar{\varphi( \bar z)}$.
Taking adjoint in the first equation of \eqref{eqnweakhomFB2}, we see that $X_{\varphi}$ satisfies $T_0^* X_{\varphi}^*=X_{\varphi}^*\hat{\varphi}(T_0^*)$. Since $K^{(\lambda)}$ is sharp, 
by Lemma \ref{weakhomlem} (or Lemma \ref{lemrosenbl}), we obtain  
$X_{\varphi}^*= M_{g_{\varphi}}C_{{\hat{\varphi}}^{-1}}$ for some non-vanishing function $g_{\varphi}$ in $ {\rm Hol}(\mathbb D)$. Furthermore, since $C_{\hat{\varphi}}$ is bounded and invertible
on $\hll$ (see Lemma \ref{lemmobinv}), it follows 
from the boundedness and invertibility of 
$X_{\varphi}$ that the operator $M_{g_{\varphi}}$ is bounded and invertible on $\hll$. Also, since $\rm{Mult}(\hll) =H^{\infty}(\D)$, $\lambda\geq 1$, it follows that $g_{\varphi}$ must be bounded above as well as bounded below on $\D$. By the same argument, we have $Y_{\varphi}^*=M_{h_{\varphi}}C_{{\hat{\varphi}}^{-1}}$ for some non-vanishing function $h_{\varphi}$ in $ {\rm Hol}(\mathbb D)$ which is bounded above as well as bounded below on $\D$. Taking adjoint in the last equation of \eqref{eqnweakhomFB2}, we see that
$$M_{\psi}X_{\varphi}^*+T_1^*Z_{\varphi}^*=Z_{\varphi}^*\hat{\varphi}(T_0^*)+Y_{\varphi}^*{\widehat{\varphi^\prime}(T_1^*)}M_{\psi}.$$
Equivalently,
\begin{align*}
Z_{\varphi}^*\hat{\varphi}(T_0^*)-T_1^*Z_{\varphi}^*&=M_{\psi}X_{\varphi}^*-Y_{\varphi}^*{\widehat{\varphi^\prime}(T_1^*)}M_{\psi}\\
&=M_{\psi}M_{g_{\varphi}}C_{{\hat{\varphi}}^{-1}} -M_{h_{\varphi}}C_{{\hat{\varphi}}^{-1}}M_{\widehat{\varphi^\prime}}M_{\psi}\\
&=M_{\ell_{\varphi}}C_{{\hat{\varphi}}^{-1}},
\end{align*}
where $\ell_{\varphi}=\psi g_{\varphi}-h_{\varphi}(\widehat{\varphi^\prime}\circ {\hat{\varphi}}^{-1})(\psi\circ {\hat{\varphi}}^{-1})$.
Since the kernel $K^{(\lambda)}$ is sharp, by Lemma \ref{lemrosenbl}, it follows that 
\begin{equation}
Z_{\varphi}^*f=\ell_{\varphi}(\hat{\varphi}^{-1})^{\prime}(f^\prime\circ \hat{\varphi}^{-1})
+\chi_{\varphi}(f\circ \hat{\varphi}^{-1}),\;\;f\in \hll,
\end{equation}
for some $\chi_{\varphi}\in \mbox{ Hol}(\D)$.
Furthermore, since the composition operator $C_{\hat{\varphi}}$ is bounded on $\hlm$ by Lemma \ref{lemmobinv}, the operator 
$C_{\hat{\varphi}}Z_{\varphi}^*$ is bounded from $\hll$ to $\hlm$. Note that 
$$C_{\hat{\varphi}}Z_{\varphi}^*(f)=(\ell_{\varphi}\circ \hat{\varphi})((\hat{\varphi}^{-1})^{\prime}\circ \hat{\varphi})f^{\prime}+(\chi_{\varphi}\circ \hat{\varphi})f,\;\;f\in \hll.$$
Since $\mu<\lambda+2$, by Lemma \ref{lemstrongir}, it follows that that $(\ell_{\varphi}\circ \hat{\varphi})((\hat{\varphi}^{-1})^{\prime}\circ \hat{\varphi})$ is identically the zero function for each $\varphi\in \mb$, and therefore, $\ell_{\varphi}$ is identically the zero function for each $\varphi\in\mb$.
 Equivalently,
\begin{equation}\label{eqnnonweaklyhomo}
\psi(z)g_{\varphi}(z)=h_{\varphi}(z) (\widehat{\varphi^\prime}\circ {\hat{\varphi}}^{-1})(z) (\psi\circ {\hat{\varphi}}^{-1})(z),
\;\;z\in \D, \varphi\in \mb. 
\end{equation} 
First, we show that $\psi$ is non-vanishing on ${\D}$.
If possible let $\psi(w_0)=0$ for some $w_0\in \D$, and let $w$
be a fixed but arbitrary point in $\D$. 
By transitivity of $\mb$, there exists a function $\varphi_w$ in $\mb$ such that ${\widehat{\varphi_{\!w}}}^{-1}(w_0)=w$. Putting $z=w_0$ and $\varphi=\varphi_w$ in \eqref{eqnnonweaklyhomo}, we see that
\begin{equation*}
\psi(w_0)g_{\varphi_w}(w_0)=h_{\varphi_w}(w_0) (\widehat{\varphi_w^\prime}\circ {\widehat{\varphi_w}}^{-1})(w_0) \psi(w).
\end{equation*}
Since the functions $h_{\varphi_w}$ and $(\widehat{\varphi_w^\prime}\circ {\widehat{\varphi_w}}^{-1})$ are non-vanishing on $\D$, 
it follows from the above equality that $\psi(w)= 0$. Since this holds for an arbitrary $w\in \D$, we conclude that $\psi$ vanishes on $\bar{\D}$,   
which contradicts that $\psi$ is non-zero on $\bar{\D}$. Hence $\psi$ is non-vanishing on $\D$.

Now we  show that $\psi$ is non-vanishing on the unit circle $\mathbb T$. Replacing $\varphi$ by $\varphi_{\theta,0}$ (which is the rotation map $e^{i\theta}z$) in \eqref{eqnnonweaklyhomo},
we obtain
\begin{equation}\label{eqnthm2.3.15}
\psi(z)g_{\varphi_{\theta,0}}(z)=e^{-i\theta}h_{\varphi_{\theta,0}}(z)\psi(e^{i\theta}z),\;\; z\in \D.
\end{equation}
Let $\{w_n\}_{n\geq 0}$ be a sequence in $\D$ such that $w_n\to 1$ as $n\to \infty$. If possible let $\psi$ vanishes at some point $e^{i\theta_0}$ on $\mathbb T$. Putting $z=e^{i\theta_0}w_n$ in \eqref{eqnthm2.3.15}, we obtain
\begin{equation}
\psi(e^{i\theta_0}w_n)g_{\varphi_{\theta,0}}(e^{i\theta_0}w_n)=e^{-i\theta}h_{\varphi_{\theta,0}}(e^{i\theta_0}w_n)\psi(e^{i(\theta_0+\theta)}w_n)\;\;\mbox{for all}\;\; n\geq 0.
\end{equation}
Since $\psi\in C(\bar{\D})$ and $g_{\varphi_{\theta,0}}, h_{\varphi_{\theta,0}}$
 are bounded above as well as bounded below on $\D$, taking limit as $n\to \infty$, it follows that $\psi(e^{i(\theta_0+\theta)})=0$. Since this is true for any $\theta\in \mathbb R$, we conclude that $\psi$ vanishes at all points on $\mathbb T$. Consequently, $\psi$ is identically zero on $\bar{\D}$. This contradicts our hypothesis that $\psi$ is non-zero on $\bar{\D}$. This completes the proof.
\end{proof}

As an immediate consequence of the above theorem, we obtain a class of operators in $\mathcal FB_2(\D)$ which are not weakly homogeneous.

\begin{corollary}
Let $1\leq\lambda\leq \mu<\lambda+2$. If $\psi$ is a non-zero function in $C(\bar{\mathbb D})\cap \rm Hol(\mathbb D)$ with atleast one zero in $\bar{\D}$, then the operator $T=\begin{pmatrix}
M_z^* & M_\psi^*\\
0 & M_z^*
\end{pmatrix}
$ on $\hll\oplus\hlm$ is not weakly homogeneous.
\end{corollary}

As a consequence of the Lemma \ref{lemstrongir}, we also obtain the following proposition which is a strengthening of \cite[Theorem 4.5 (2)]{quasihomogeneous} in the particular case of quasi-homogeneous operators of rank $2$. 
Recall that an operator $T$ is said to be strongly irreducible if $XTX^{-1}$ is irreducible for all invertible operator $X$.
\begin{proposition}
Let $0< \lambda\leq \mu<\lambda+2$ and $\psi \in {\rm Mult} (\hll, \hlm)$. Let $T=\begin{pmatrix}
M_z^* & M_\psi^*\\
0 & M_z^*
\end{pmatrix}
$ on $\hll\oplus \hlm$.
If $\psi$ is non-zero, then $T$ is strongly irreducible.
\end{proposition}
\begin{proof}
Suppose that $\psi$ is non-zero and $T$ is not strongly irreducible.
Then, by \cite[Proposition 2.22]{Flagstructure}, there exists a bounded operator $X:\hlm\to \hll$ such that
$ X^*T_0^*-T_1^*X^*= M_{\psi}.$ Since the kernel  $K^{(\lambda)}$ is sharp, by Lemma $\ref{lemrosenbl}$ (with $\varphi$ to be the identity map), there exists a function 
$\chi\in {\rm Hol}(\D)$ such that $X^*(f)= \psi f'+\chi f, f\in \hll.$ Since $X$ is bounded and $ \mu < \lambda+2$, by Lemma $\ref{lemstrongir},$ $\psi$ is identically zero on $\D.$ This is a  contradiction to the assumption  that $\psi$ is non-zero. Hence $T$ must be strongly irreducible, completing the proof.
\end{proof}

\section{M\"obius bounded operators}\label{secmob}
In this section, we find some necessary conditions for M\"obius boundedness of the multiplication operator $M_z$ on the reproducing kernel Hilbert space 
$(\hl, K)$, where $K(z,w)$ is of the form $\sum_{n=0}^\infty b_n (z\bar{w})^n$, $b_n>0$, on $\D\times \D$. As a consequence, we show that the multiplication operator $M_z$ on the  Dirichlet space is not M\"obius bounded. We begin with a preparatory lemma.

First we note that 
the power series representation of the biholomorphic automorphism $\ft$ 
is given by $\sum_{n=0}^\infty \alpha_n z^n$, $z\in \D$, where 
\begin{equation}\label{powerserfi}
\alpha_0=-e^{i\theta}a\;\;\mbox{and}\;\;\;\; \alpha_n=e^{i\theta}(1-|a|^2)(\bar{a})^{n-1},\;\;\; n\geq 1. 
\end{equation}

\begin{lemma}\label{Mbprop1}
Let $K(z,w)=\sum_{n=0}^\infty b_n(z\bar{w})^n$, $b_n>0$, be a positive definite kernel on $\D\times \D$. 
Suppose that the multiplication operator $M_z$ is bounded on 
$(\hl, K)$ and $\sigma(M_z)=\bar{\D}$. If the sequence $\left\{n b_n\right\}_{n\in\mathbb Z_+}$ 
is bounded, then there exists a constant $c>0$ such that 
$$\|\varphi_{\theta, a}(M_z)\| \geq \frac{K(a,a)}{c}-|a|\;\;\mbox{for all}\; \;a\in \D, \;\theta\in [0,2\pi).$$
\end{lemma}
 
\begin{proof}
Since $\{n b_n\}$ is bounded, there exists a constant $c>0$ such that $n b_n<c$ for all $n\geq 0$. For $a\in \D, \theta\in [0,2\pi)$, setting
$\tilde{\varphi}_{\theta, a}(z)=\ft(z)-\ft(0),\; z\in \D$, and using \eqref{powerserfi}, we see that 
\begin{equation}
\tilde{\varphi}_{\theta, a}(z) K(z,a)= \sum_{n=1}^\infty
\left(\sum_{k=1}^n \alpha_k b_{n-k}(\bar{a})^{n-k}\right)z^n,~z\in  \mathbb D.
\end{equation}

By hypothesis, $\varphi_{\theta, a}(M_z)$ is bounded, and hence $\tilde{\varphi}_{\theta, a}(\cdot) K(\cdot,a)$ belongs to $(\hl, K)$. Note that 
\begin{align}\label{Mbprop2}
\begin{split}
\|\tilde{\varphi}_{\theta, a}(\cdot)K(\cdot,a)\|^2 &= \sum_{n=1}^\infty \Big|
\sum_{k=1}^n \alpha_k b_{n-k}(\bar{a})^{n-k}\Big |^2 \|z^n\|^2\\
&=(1-|a|^2)^2\sum_{n=0}^\infty|a|^{2n}
\big(\sum_{j=0}^{n}b_j\big)^2\frac{1}{b_{n+1}}.
\end{split}
\end{align}

\textbf{Claim}: For all $a\in\mathbb D$, the following inequality holds:
\begin{equation}\label{eqnclaim}
\sum_{n=0}^\infty|a|^{2n}\big(\sum_{j=0}^{n} b_j\big)^2
\frac{1}{b_{n+1}}\geq \frac{1}{c}(1-|a|^2)^{-2} K(a,a)^2.
\end{equation}

Since $(1-|a|^2)^{-2}=\sum_{n=0}^\infty (n+1)|a|^{2n}$, $a\in \D$, setting $\beta_n=\sum_{j=0}^n (j+1)b_{n-j}$, $n\geq 0$, we see that 
$$(1-|a|^2)^{-2}K(a,a)=\sum_{n=0}^\infty \beta_n |a|^{2n},\;\; a\in \D.$$
Furthermore, setting $\gamma_n=\sum_{j=0}^n \beta_j b_{n-j}$, $n\geq 0$, we see that
\begin{equation}\label{eqnmob3}
(1-|a|^2)^{-2}K(a,a)^2=\sum_{n=0}^\infty {\gamma}_n |a|^{2n},\; \;a\in \D.
\end{equation} 
Note that
 $$\beta_n=\sum_{j=0}^n(j+1)b_{n-j}\leq (n+1)\big(\sum_{j=0}^n b_j\big),\;\;n\geq 0.$$
Therefore
$$\gamma_n=\sum_{j=0}^n \beta_j b_{n-j}
\leq \sum_{j=0}^n (j+1)\big(\sum_{p=0}^j b_p\big)b_{n-j}
\leq (n+1)\big(\sum_{j=0}^n b_j\big)^2.$$
Consequently,
\begin{align*}
\sum_{n=0}^\infty \gamma_n |a|^{2n}&\leq \sum_{n=0}^\infty (n+1)\big(\sum_{j=0}^n b_j\big)^2|a|^{2n}\\
&\leq c\sum_{n=0}^\infty|a|^{2n} \big(\sum_{j=0}^n b_j\big)^2\frac{1}{b_{n+1}},
\end{align*}
where for the last inequality, we have used that $nb_n< c$, $n\geq 0$.
Hence, by \eqref{eqnmob3}, the claim is verified.

Combining the claim with $(\ref{Mbprop2})$, it follows that 
$$\|\tilde{\varphi}_{\theta, a}(\cdot)K(\cdot,a)\|^2\geq
\frac{1}{c}K(a,a)^2.$$ 
Since $\|K(\cdot,a)\|^2=K(a,a)$, it follows that 
$$\|\tilde{\varphi}_{\theta, a}(M_z)\|^2 \geq \frac{\|\tilde{\phi}_{\theta, a}(\cdot)K(\cdot,a)\|^2}
{\|K(\cdot,a)\|^2}\geq \frac{1}{c}K(a,a).$$

Finally, note that for $a\in \D$,
$$ \|\varphi_{\theta, a}(M_z)\| =\|\tilde{\varphi}_{\theta, a}(M_z)+{\varphi}_{\theta, a}(0)I\|\geq
\|\tilde{\varphi}_{\theta, a}(M_z)\|-|a|\geq\frac{1}{c}K(a,a)-|a|.$$
This completes the proof.
\end{proof}
The following lemma, which  is the easy half of the statement of \cite[Lemma 2]{similaritywtshifts}, will be used later in this section.  

\begin{lemma}\label{lemMobnec1}
Let $f(x)=\sum_{n=0}^\infty a_n x^n$, $0\leq x<1$, where  $a_n\geq 0$. If $f(x)\leq  c(1-x)^{-\alpha}$, $ 0\leq x <1 $, for some constants $\alpha, c>0$, then there exists  $c^\prime >0$ such that
$$a_0+a_1+...+a_n\leq c^\prime (n+1)^{\alpha}\;\;\;\mbox{for all}~\;n\geq 0.$$
\end{lemma}
The following lemma will be used in the proof of  Theorem \ref{thmMobnec1}.
\begin{lemma}[cf. \cite{Hardyinequality}]\label{lemMobnec2}
If $\{b_n\}_{n\in \mathbb Z_+}$ is a sequence of positive real numbers such that $\sum_{n=0}^\infty b_n <\infty$,  
then 
$$ \sum_{n=0}^\infty \frac{n+1}{\frac{1}{b_0}+\frac{1}{b_1}+...+\frac{1}{b_{n}}}\leq 2\sum_{n=0}^\infty b_n.$$
\end{lemma}

\begin{theorem}\label{thmMobnec1}
Let $ K(z,w)= \sum_{n=0}^\infty  b_n(z\bar{w})^n$, $b_n>0$, be a positive definite kernel on $\D\times \D$. If the multiplication operator $M_z$ on 
$(\mathcal H, K)$ is M\"obius bounded, then
\begin{itemize}
\item[\rm (i)] $\sum_{n=0}^\infty b_n=\infty$, and
\item[\rm (ii)] the sequence $ \{nb_n\}_{n\in\mathbb Z_+}$ 
is unbounded.
\end{itemize}
\end{theorem}

\begin{proof} \rm (i) Note that for any $\theta\in [0, 2\pi)$, $a\in \mathbb D$ and 
$j \in \mathbb Z_+$, we have
\begin{align}\label{eqnmob7}
\begin{split}
\left\|\ft(M_z)\left(\frac{z^j}{\|z^j\|}\right)\right \|^2
& =\frac{1}{\|z^j\|^2}\|\ft(z)z^j\|^2\\
& = \frac{1}{\|z^j\|^2}\big(|a|^2\|z^j\|^2+(1-|a|^2)^2\sum_{n=1}^\infty |a|^{2(n-1)}
\|z^{n+j}\|^2\big).
\end{split}
\end{align}
If $M_z$ on $(\hl, K)$ is M\"obius bounded, then there exists a constant $c>0$ such that
$$\sup_{\theta\in[0,2\pi), a\in \mathbb D, j\in  \mathbb Z_+}\left\|\ft(M_z)\left(\frac{z^j}{\|z^j\|}\right)\right \|^2\leq c.$$
Therefore, from \eqref{eqnmob7}, we see that
$$(1-|a|^2)^2 \sum_{n=1}^\infty |a|^{2(n-1)}\frac{\|z^{n+j}\|^2}{\|z^j\|^2}\leq 
c,\;\;a\in \mathbb D, j\in \mathbb Z_+.$$
Replacing $|a|^2$ by $x,$ we obtain 
\begin{equation}\label{eqnthmMobnec1}
\sum_{n=0}^\infty c_{n,j}x^n \leq \frac{c}{(1-x)^2},\;\;x\in [0,1),
\end{equation}
where $c_{n,j}=\frac{\|z^{n+j+1}\|^2}{\|z^j\|^2}$, $n,j\in \mathbb Z_+$.
Hence, applying Lemma $\ref{lemMobnec1},$ we see that there
exists a constant $c^\prime>0$ such that for all $n, j$ $\in \mathbb Z_+,$
$$( c_{0,j}+c_{1,j}+\cdots+ c_{n,j}) \leq c^\prime (n+1)^2.$$
Since $b_n=\frac{1}{\|z^n\|^2}$, $n\in \mathbb Z_+$, putting $j=0$ in the above inequality, we obtain 
$$\left(\frac{1}{b_1}+\frac{1}{b_2}+\cdots+\frac{1}{b_{n+1}}\right)
\leq \frac{c^\prime}{b_0}(n+1)^2,\; n\in \mathbb Z_+.$$
Therefore
$$\sum_{n=0}^\infty \frac{n+1}{\frac{1}{b_1}+\frac{1}{b_2}+\cdots+\frac{1}{b_{n+1}}}\geq \frac{b_0}{c^\prime}\sum_{n=0}^\infty \frac{1}{n+1}.
$$
Consequently, $\sum_{n=0}^\infty \frac{n+1}{\frac{1}{b_1}+\frac{1}{b_2}+\cdots+\frac{1}{b_{n+1}}}=\infty$.
Hence, by Lemma $\ref{lemMobnec2},$ we conclude that
$\sum_{n=0}^\infty b_n 
= \infty.$\\
\rm (ii) Suppose that $M_z$ on $(\hl, K)$ is M\"obius bounded, and if possible, let  $\{nb_n \}_{n\in \mathbb Z_+}$ is bounded.
Then by Lemma $\ref{Mbprop1},$ there exists a constant $c>0$ such that 
$$\sup_{a\in \mathbb D}\Big(\;\frac{K(a,a)}{c}-|a|\;\Big)< \infty.$$
Therefore $\sup_{a\in \D}K(a,a)$ is also finite.
Since Abel summation method is totally regular (see\cite[page 10]{Divergentseries}), it follows that $\sum_{n=0}^\infty b_n$ is finite. By part $\rm(i)$ of this Theorem, this  is a contradiction to the assumption that $M_z$ is M\"obius bounded. Hence the sequence 
$\big\{n b_n\big\}_{n\in \z}$ is unbounded, completing the proof. 
\end{proof}

\begin{corollary}
Let $K(z,w)=\sum_{n=0}^\infty b_n(z\bar{w})^n$, $b_n>0$, be a positive definite kernel on $\D\times \D$. Suppose that  $b_n\sim (n+1)^{\gamma}$ for some  $\gamma\in\mathbb R$. 
Then the multiplication operator $M_z$ on $(\mathcal{H}, K)$ is M\"obius bounded if and only if $\gamma > -1.$ In particular, the multiplication operator $M_z$ on the Dirichlet space is not M\"obius bounded.
\end{corollary}
\begin{proof}
It follows from \cite[Theorem $2^\prime$]{Shields} that  $M_z$ on $(\hl, K)$ is similar to the operator $M_z$ on $(\hl, K_{(\gamma)})$. 
Since similarity preserves M\"obius boundedness, it suffices to show that $M_z$ on $(\hl, K_{(\gamma)})$ is M\"obius bounded if and only if $\gamma>-1$. If $\gamma>-1$, then by Corollary \ref{cormobinv}, $M_z$ on $(\hl, K_{(\gamma)})$
is similar to a homogeneous operator, and therefore is  M\"obius bounded. 
If $\gamma\leq -1$, then note that the sequence $\{n.(n+1)^{\gamma}\}_{n\in \z}$ is bounded. Hence, by Theorem \ref{thmMobnec1} \rm (ii),  $M_z$ on $(\hl, K_{(\gamma)})$ is not M\"obius bounded.  
\end{proof}

%

The following theorem shows that Shields' conjecture has an affirmative answer in a smaller class of weighted shifts containing the non-contractive homogeneous operators in $B_1(\D)$.
\begin{theorem}

Let $K(z,w)= \sum_{n=0}^\infty  b_n(z\bar{w})^n$ be a positive 
definite kernel on $\D \times \D$. Assume that the sequence $\big\{b_n\big\}_{n\in \mathbb Z_+}$ is decreasing. If the multiplication operator $M_z$ on 
$(\mathcal H, K) $ is M\"obius bounded, then there exists a constant $ c>0 $ such that
$\|M_z^n\|\leq c(n+1)^{\frac{1}{2}}$ for all $n\in \z.$
\end{theorem}

\begin{proof}
It suffices to show that $\|M_z^{n+1}\|\leq c(n+1)^{\frac{1}{2}},\;\;n\in \z$.
By a straightforward computation, we see that
\begin{equation}\label{eqnshldconj}
\|M_z^{n+1}\|^2=\sup_{j\in \z}\frac{\|z^{n+j+1}\|^2}{\|z^j\|^2},\;\;n\in \z.
\end{equation}
From $\eqref{eqnthmMobnec1},$ we already have that 
$$\sum_{n=0}^\infty c_{n,j}x^n \leq \frac{c}{(1-x)^2},\;\;
x\in [0,1),$$
where $c_{n,j}=\frac{\|z^{n+j+1}\|^2}{\|z^j\|^2}.$ Multiplying both sides 
by $1-x$, we see that 
$$c_{0,j}+\sum_{n=1}^\infty (c_{n,j}-c_{n-1,j})x^n\leq \frac{c}{1-x},\;\;x\in [0,1), 
j \in \mathbb Z_+.$$
Since $\|z^n\|^2=\frac{1}{b_n}$ and $\{b_n\}_{n\in \mathbb Z_+}$ is decreasing, the sequence $\{c_{n,j}\}_{n \in \mathbb Z_+}$ is increasing. 
Consequently, $(c_{n,j}-c_{n-1,j})\geq 0$ for all $n\geq 1, j\geq 0$.
Therefore, using Lemma $\ref{lemMobnec1},$ we conclude that
there exists a constant $c^\prime >0$ (independent of $n$ and $j$) such that 
$$c_{0,j}+(c_{1,j}-c_{0,j})+\cdots+(c_{n,j}-c_{n-1,j})\leq c^\prime (n+1)\;\;\mbox{for all} \;\; n,j \in \mathbb Z_+.$$
Equivalently,
$$c_{n,j}\leq c^\prime (n+1)\;\;\mbox{for all} \;\; n,j \in \mathbb Z_+.$$
Hence, in view of \eqref{eqnshldconj}, we conclude that
$\|M_z^{n+1}\|^2\leq c^\prime (n+1)$ for all $n\in \z$, completing the proof.
\end{proof}
\subsection{M\"obius bounded quasi-homogeneous operators}
In this subsection we identify all quasi-homogeneous operators which are M\"obius bounded. We start with the following theorem which gives a necessary condition for a class of operators in $\mathcal FB_2(\D)$ to be M\"obius bounded. 

\begin{theorem}\label{thmquasmob}
Let $0< \lambda\leq \mu$, and let $\psi$ be a non-zero function in ${\rm Mult} (\hll, \hlm)$. Let $T$ be the operator $\begin{pmatrix}
M_z^* & M_\psi^*\\
0 & M_z^*
\end{pmatrix}
$ on $\hll\oplus \hlm$. If $T$ is M\"obius bounded, then $\mu-\lambda \geq 2.$
\end{theorem}

\begin{proof}
It suffices to show that if 
$T^*$ is M\"obius bounded, then $\mu-\lambda \geq 2.$
Since $\sigma (M_z) = \bar{\mathbb D}$ on both $\hll$ and $\hlm,$ 
it is easily verified that $\sigma(T)=\bar{\mathbb D}$.
As before, for $\varphi\in \mb$, we have
$$\varphi(T^*)= \begin{pmatrix}
M_{\varphi}& 0\\
M_{\psi\varphi^\prime} & M_{\varphi}
\end{pmatrix} \mbox{~~on}~~ \hll\oplus\hlm.$$
Observe that for an operator of the form 
$\begin{psmallmatrix}
A&0\\
B&C
\end{psmallmatrix}$, we have 
$\|B\|\leq \left\|\begin{psmallmatrix}
A&0\\
B&C
\end{psmallmatrix}\right\|\leq (\|A\|+\|B\|+\|C\|).$
Therefore, we have
\begin{equation*}
\|M_{\psi\varphi^\prime}\|_{\hll\to\hlm}\leq \|\varphi(T^*)\|\leq \|M_{\varphi}\|_{\hll}+\|M_{\varphi}\|_{\hlm}+ \|M_{\psi\varphi^\prime}\|_{\hll\to\hlm}.
\end{equation*}
Since the multiplication operator $M_z$ on $\hll$ as well as on $\hlm$ is M\"obius bounded, it follows from the above inequality that $T^*$ is M\"obius bounded if and only if $\sup_{\varphi\in \mb}\|M_{\psi\varphi'}\|_{\hll\to\hlm}$ is finite.
Now for all $w$ in $\D$, we have
\begin{align*}
\|M_{\psi\varphi^\prime}\|_{\hll\to\hlm}^2=\|(M_{\psi\varphi^\prime})^*\|_{\hlm\to\hll}^2& \geq 
\frac{\|(M_{\psi\varphi^\prime})^*(K^{(\mu)}(\cdot,w))\|^2}{\|K^{(\mu)}(\cdot,w)\|^2}\\
&=|\psi(w)\varphi^\prime(w)|^2\frac{\|K^{(\lambda)}(\cdot,w)\|^2}{\|K^{(\mu)}(\cdot,w)\|^2}\\
&=|\psi(w)\varphi^\prime(w)|^2 (1-|w|^2)^{\mu-\lambda}.
\end{align*}

Note that $\ft^\prime(w)=e^{i\theta}\frac{1-|a|^2}{(1-\bar{a}w)^2}$, $w\in \D$.
Thus, if $T^*$ is M\"obius bounded, then there exists 
a constant $c>0$ such that
$$\sup_{a,w\in \mathbb{D}}\frac{|\psi(w)|^2(1-|a|^2)^2}{|1-
\bar a w|^4} (1-|w|^2)^{\mu-\lambda}\leq c .$$ 
Taking $a=w$, we obtain
\begin{equation}\label{eqnmob_bnd}
|\psi(w)|^2 \leq c(1-|w|^2)^{-(\mu-\lambda-2)},\;\;w\in \D.
\end{equation}
If possible, assume that $\mu-\lambda-2<0$. Then, by an application of maximum modulus principle, it follows from \eqref{eqnmob_bnd} that $\psi$ is identically zero, which is a contradiction to our assumption that $\psi$ is non-zero. Hence $\mu-\lambda \geq 2.$ 
\end{proof}

\subsubsection{Quasi-homogeneous operators}
Let $n\geq 1$ be a positive integer, and let  $0<\lambda_0\leq \lambda_1\leq \ldots\leq \lambda_{n-1}$ be $n$ positive real numbers such that the difference $\lambda_{i+1}-\lambda_i$ is a fixed number $\Lambda$ for all $0\leq i\leq n-2.$ 
Let $T_i$, $0\leq i\leq n-1$, denote the adjoint $M_z^*$ of the multiplication operator by the coordinate function $z$ on $\hl^{(\lambda_i)}$. Furthermore, let $S_{i,j}, 0\leq i < j \leq n-1$, be the linear map given by the formula
$$S_{i,j}(K^{(\lambda_j)}(\cdot,w))=m_{i,j} {\bar \partial}^{(j-i-1)}
K^{(\lambda_i)}(\cdot, w),$$
where  $m_{i,j}$'s are arbitrary complex numbers. Note that if $S_{i,j}$ defines a bounded 
linear operator from $\hl^{(\lambda_j)}$
to $\hl^{(\lambda_{i})}$, then $(S_{i,j})^*(f)=\overbar {m}_{i,j}f^{(j-i-1)},\;f\in \hl^{(\lambda_{i})}$.

A quasi-homogeneous operator $T$ of rank $n$ (see \cite{quasihomogeneous}) is a bounded linear operator of the form
\begin{equation}\label{eqnquasihomo}
\begingroup
\begin{pmatrix}
\renewcommand*{\arraystretch}{1.5}
T_0& S_{0,1}& S_{0,2}&\cdots&S_{0, n-1}\\
0&T_1& S_{1,2}&\cdots & S_{1,n-1}\\
\vdots&\ddots&\ddots&\ddots&\vdots\\
0&\cdots&0&T_{n-2}&S_{n-2,n-1}\\
0&0&\cdots&0&T_{n-1}
\end{pmatrix}
\endgroup
\end{equation}
on $\mathcal H^{(\lambda_0)}\oplus \mathcal H^{(\lambda_1)}\oplus 
\cdots \oplus \mathcal H^{(\lambda_{n-1})}.$ 
It is known that  the class of  quasi-homogeneous operators of rank $n$ contains all homogeneous operators in $B_n(\D)$. For a quasi-homogeneous operator $T$, let $\Lambda(T)$ denote the fixed difference $\Lambda$. 
When $\Lambda(T)\geq 2$, a repeated application of  Lemma \ref{lemdiffbdd} shows that each $S_{i, j}$, $0\leq i < j \leq n-1$, is bounded from $\hl^{(\lambda_i)}$ to $\hl^{(\lambda_j)}$, and consequently, an operator of the form \eqref{eqnquasihomo} is also bounded. In case of $\Lambda(T)<2$, the boundedness criterion for $T$ was obtained in terms of $\Lambda(T),n$ and $m_{i,j}$'s in \cite[Proposition 3.2]{quasihomogeneous}.

 It is easily verified that
a quasi-homogeneous operator $T$ satisfies $T_iS_{i,i+1}=S_{i,i+1}T_{i+1}$, $0\leq i\leq n-2$. Therefore $T$ belongs to the class $\mathcal FB_{n}(\D)\subseteq B_{n}(\D)$ (see \cite{Flagstructure}).

The theorem given below describes all quasi-homogeneous operators which are M\"obius bounded.
\begin{theorem}\label{thmmobqu}
A quasi-homogeneous operator $T$ is M\"obius bounded if and only if $\Lambda(T)\geq 2.$
\end{theorem}
\begin{proof}
If $\Lambda(T)\geq 2$, then by \cite[Theorem 4.2 (1)]{quasihomogeneous}, $T$ is similar to the direct sum
$T_0\oplus T_1\oplus\cdots \oplus T_{n-1}.$ Hence $T$ is M\"obius bounded if and only if $T_0\oplus T_1\oplus\cdots \oplus T_{n-1}$ is M\"obius bounded. Note that each $T_i$, $0\leq i\leq n-1$, is  homogeneous and therefore is M\"obius bounded. Consequently, the operator $T_0\oplus T_1\oplus\cdots \oplus T_{n-1}$ is also M\"obius bounded.

To prove the converse, assume that $T$ is M\"obius bounded. 
By a straightforward computation using the intertwining relation $T_i S_{i,i+1}=S_{i,i+1}T_{i+1}$, $0\leq i\leq n-2$, we obtain
\begin{equation*}
\begingroup
\renewcommand*{\arraystretch}{1.3}
\varphi(T)=
\begin{pmatrix}
\varphi(T_0)&S_{0,1}\varphi^\prime(T_1)& *&\cdots&*\\
0&\varphi(T_1)&S_{1,2}\varphi^\prime(T_2)& \cdots & *\\
\vdots&\ddots&\ddots&\ddots&\vdots\\
0&\cdots&0&\varphi(T_{n-2})&S_{n-2,n-1}\varphi^\prime(T_{n-1})\\
0&0& \cdots&0&\varphi(T_{n-1})
\end{pmatrix}
\endgroup
\end{equation*}
on $\mathcal H^{(\lambda_0)}\oplus \mathcal H^{(\lambda_1)}\oplus 
\cdots \oplus \mathcal H^{(\lambda_{n-1})}.$
Since 
$$\|\varphi(T)\|\geq 
\left \|\begin{pmatrix}
\varphi(T_0)&S_{0,1}\varphi^\prime(T_1)\\
0&\varphi(T_1)
\end{pmatrix}\right\| =\left\|
\varphi \begin{pmatrix}
T_0&S_{0,1}\\
0&T_1
\end{pmatrix}\right\|,$$
it follows that the operator $\begin{pmatrix}
T_0&S_{0,1}\\
0& T_1
\end{pmatrix}$
is M\"obius bounded.
Note that this operator is of the form
$\begin{pmatrix}
M_z^*&M_{\psi}^*\\
0 &M_z^*
\end{pmatrix}$ on $\mathcal H^{(\lambda_0)}\oplus \mathcal H^{(\lambda_1)}$,
where $\psi$ is the constant function $\overbar{{m}_{0,1}}.$ Hence, by Theorem $\ref{thmquasmob},$ we conclude that 
$\lambda_1-\lambda_0\geq 2.$ Consequently, $\Lambda(T)\geq 2$,
completing the proof.
\end{proof}

\begin{corollary}
The Shields' conjecture has an affirmative answer for the class of quasi-homogeneous operators.
\end{corollary}
\begin{proof}First note that, by Theorem 
\ref{thmmobqu}, a quasi-homogeneous operator $T$  is M\"{o}bius  bounded if and only if $\Lambda(T) \geq 2$. Second,  if $\Lambda(T) \geq 2$, then by \cite[Theorem 4.2 (1)]{quasihomogeneous}, $T$ is  similar to $T_0\oplus T_1\oplus \cdots\oplus T_{n-1}$.  Shields' conjecture is easily verified for these operators using the explicit weights (see \cite[section 7.2]{homogeneoussurvey}). Therefore, its validity for $T$ follows via the similarity. 
\end{proof} 

\begin{corollary}
A quasi-homogeneous operator $T$ is M\"obius bounded if and only if it is 
similar to a homogeneous operator.
\end{corollary}
\begin{proof}The proof in the forward direction is exactly the same as the proof given in the previous corollary.  In the other direction, an operator similar to a homogeneous operator is clearly M\"{o}bius bounded. 
\end{proof}

The corollary given below follows immediately from Proposition \ref{thmweakhominFB_2}.
\begin{corollary}
Every quasi-homogeneous operator $T$ of rank $2$ is weakly homogeneous.
\end{corollary}
\section{A M\"obius bounded weakly homogeneous operator which is\\ not similar to any homogeneous operator}\label{secexamplehomo}
In this section, we provide a class of examples to show that a
M\"{o}bius bounded weakly homogeneous operator need not be similar to any homogeneous operator. The lemma given below is undoubtedly well-known (for a proof, see \cite[Corollary 2.37]{pick-int}). 
\begin{lemma}\label{lembounded}
Let  $K:\D \times \D \to \mathbb C$ be a positive definite kernel, and let $f:\D \to \mathbb C$ be an arbitrary 
holomorphic function. Then the operator $M_f$ of multiplication by $f$  on $(\mathcal H, K)$ is bounded if and only if 
there exists 
a  $c >0$ such that $\big (c^2 -f(z)\overbar{f(w)} \big) K(z,w)$ is non-negative definite on $\D\times \D$. 
In case $M_f$ is bounded, $\|M_f\|$ is the infimum of all
$c>0$ such that $\big (c^2 -f(z)\overbar{f(w)} \big ) K(z,w)$ is non-negative definite.
\end{lemma}

The following lemma, which will be used in the proof of the main theorem of this section, provides a sufficient condition on $K$ to determine if the multiplication operator $M_z$ on $(\hl, K)$ is M\"obius bounded.
\begin{lemma}\label{lemexample}
Let $K:\mathbb D\times \mathbb D\to\mathbb C$ be a 
positive definite kernel. Suppose that 
$K$ can written as the product $K^{(\lambda)}\tilde{K}$, where $\lambda > 0$ and $\tilde{K}$ is some non-negative definite kernel on $\D\times \D$. Then the multiplication operator $M_z$ on 
$\hk$ is bounded and $\sigma(M_z)=\bar{\mathbb D}.$ 
Moreover, $M_z$ on $\hk$ is M\"obius bounded.
\end{lemma}

\begin{proof}
Let $M^{(\lambda)}$ denote the multiplication operator $M_z$ on  $\mathcal H^{(\lambda)}$.
Since $M^{(\lambda)}$ is M\"obius bounded, by Lemma \ref{lembounded}, there exists a constant $c>0$ such that $(c^2-\varphi(z)\overbar{\varphi(w)})K^{(\lambda)}(z,w)$ is non-negative definite on $\mathbb D\times \mathbb D$ for all $\varphi$ in $\mb$. Hence $(c^2-\varphi(z)\overbar{\varphi(w)})K(z,w)$, being a product of two non-negative definite kernels $(c^2-\varphi(z)\overbar{\varphi(w)})K^{(\lambda)}(z,w)$ and $\tilde{K}(z,w)$, is non-negative definite. Therefore, again by 
Lemma \ref{lembounded}, it follows that $\|M_{\varphi}\|_{(\hl,K)}\leq c$ for all $\varphi \in \mb$. 

To show that the spectrum of $M_z$ on $(\hl, K)$ is $\bar{\mathbb D}$, let $a$ be an arbitrary point in $\mathbb C\setminus \bar{\D}$. Since $\sigma(M^{(\lambda)})=\bar{\mathbb D}$, the operator $M_{z-a}$ is invertible on $\mathcal H^{(\lambda)}$. Consequently, the operator $M_{(z-a)^{-1}}$ is bounded on $\mathcal H^{(\lambda)}$. Then, by the same argument used in the previous paragraph, it follows that  $M_{(z-a)^{-1}}$ is bounded on $(\mathcal H,K)$, and therefore $a\notin \sigma(M_z)$. Also, since each $K(\cdot,w)$, $w\in \D$, is an eigenvector of $M_z^*$ on $(\hl, K)$, it follows that $\bar{\mathbb D}\subseteq \sigma(M_z)$. Therefore we conclude that $\sigma(M_z)=\bar{\mathbb D}$. 

Since $\sigma(M_z)=\bar{D}$ and  $\|M_{\varphi}\|_{(\hl,K)}\leq c$ for all $\varphi \in \mb$, it follows that $M_z$ on $(\hl, K)$ is M\"obius bounded.
\end{proof}
The theorem given below answers Question \ref{quesweakhomo} in the negative.
\begin{theorem}\label{thmexample}
Let $K(z,w)=\sum_{n=0}^\infty b_n(z\bar{w})^n$, $b_n>0$, be a positive definite kernel on $\D\times \D$ such that for each $\gamma\in \mathbb R$, $\lim_{|z|\to 1}(1-|z|^2)^{\gamma}K(z,z)$ is either $0$ or $\infty$.  Assume that the adjoint $M_z^*$ of the multiplication operator by the coordinate function $z$ on $(\hl, K)$  is in $B_1(\mathbb D)$ and is weakly homogeneous. Then  the multiplication operator $M_z$ on $(\mathcal H,K K^{(\lambda)})$, $\lambda > 0$,  is a M\"obius bounded weakly homogeneous operator which is not similar to any homogeneous operator.
\end{theorem}

\begin{proof}
Since the operator $M_z^*$ on $(\hl, K)$ is weakly homogeneous, so is the operator $M_z$ on $(\hl, K)$. Furthermore, 
note that   $M_z$ on $(\mathcal H, K^{(\lambda)})$ is homogeneous,
and both the kernels $K$ and $K^{(\lambda)}$ are sharp. Therefore, 
 by Theorem $ \ref{thmweakhom}$, it follows that $M_z$ on
$(\mathcal H, K K^{(\lambda)})$ is weakly homogeneous. 
Also, by Lemma \ref{lemexample}, we see that $M_z$ on $(\mathcal H, K K^{(\lambda)})$ has spectrum $\bar{\mathbb D}$ and is M\"obius bounded. Therefore, to complete the proof, it remains to show that $M_z$ on $(\mathcal H,K K^{(\lambda)})$ is not similar to any homogeneous operator.

Suppose that $M_z$ on $(\mathcal H, K K^{(\lambda)})$ is similar to a homogeneous operator, say 
$T$. Since the operators $M_z^*$ on $(\mathcal H, K^{(\lambda)})$ and $M_z^*$ on $(\hl, K)$ belong to $B_1(\mathbb D)$, by \cite[Theorem 2.6]{Sal}, the operator $M_z^*$ on $(\mathcal H, K K^{(\lambda)})$ belongs to 
$B_1(\mathbb D)$. Furthermore, since the class $B_1(\mathbb D)$ is closed under similarity, it follows that the operator 
$T^*$ belongs to $B_1(\mathbb D)$. 
Also since $T$ is homogeneous, the operator $T^*$ is homogeneous. 
By the argument used in Corollary \ref{cormobinv}, it follows that $M_z$ on $(\mathcal H, K K^{(\lambda)})$ is similar to $M_z$ on $(\mathcal H, K^{(\mu)})$ for some $\mu>0$. Hence, by \cite[Theorem $2^\prime$]{Shields}, there exist constants $ c_1, c_2>0 $ such that
$$c_1 \leq \frac{K(z,z) K^{(\lambda)}(z,z)}{K^{(\mu)}(z,z)}\leq c_2,\;\;z\in \mathbb D.$$ 
Equivalently,
$$ c_1\leq (1-|z|^2)^{\mu-\lambda}K(z,z)\leq c_2,  z\in \mathbb D.$$
This is a contradiction to our hypothesis that for each $\gamma\in \mathbb R$,  
$\lim_{|z|\to 1}(1-|z|^2)^{\gamma}K(z,z)$ is either $0$ or $\infty$. Hence the operator $M_z$ on $(\mathcal H, K K^{(\lambda)})$ can  not be similar to any homogeneous operator, completing the proof of the theorem.
\end{proof}

Below we give one example which satisfy the hypothesis of the Theorem \ref{thmexample}. Recall that the Dirichlet kernel $K_{(-1)}$ is defined by 
$$K_{(-1)}(z,w)=\sum_{n=0}^\infty \frac{1}{n+1}(z\bar{w})^n=\frac{1}{z\bar{w}} \log\frac{1}{1-z\bar{w}},\;\;z,w\in \D.$$
By corollary \ref{cormobinv}, the operator $M_z^*$ on $(\hl, K_{(-1)})$ is weakly homogeneous and belongs to $B_1(\D)$. 
Also, for an arbitrary real number $\gamma$, $ \lim_{|z|\to 1} \frac{(1-|z|^2)^{\gamma}}{|z|^2} \log \frac{1}{1-|z|^2} $ 
is either $ 0 $ or $ \infty $, which follows from the following identity:
\begin{numcases}{ \lim_{x\to 0} {-x^{\gamma}\log x}=}
\infty & (if $\gamma\leq 0$)
\nonumber\\
0 & (if $\gamma>0$).
\nonumber
\end{numcases}
Consequently, we have the following corollary.
\begin{corollary}\label{corexam}
The multiplication operator $M_z$ on 
$(\mathcal H, K_{(-1)} K^{(\lambda)})$, $\lambda>0$,  
is a M\"obius bounded weakly homogeneous operator 
which is not similar to any homogeneous operator.
\end{corollary}
\section{Concluding remarks and open questions}
Suppose  $T\in B(\hl)$ is similar to a homogeneous operator $S$. Then it follows from \eqref{eqnsimtohomo} that 
$\varphi(T)=\Gamma(\varphi) T  {\Gamma(\varphi)}^{-1}$, where $\Gamma(\varphi) = XU_{\phi}X^{-1}$, $\varphi\in \mb$, and therefore the map $\Gamma: \mb\to B(\hl)$ is weakly uniformly bounded in the sense that 
$\sup_{\varphi\in \mb}\|\Gamma(\varphi)\|\|\Gamma(\varphi)^{-1}\|$
is finite.
Note that if $T$ is a weakly homogeneous operator such that the intertwining operator $\Gamma(\varphi)$ is weakly uniformly bounded, then $T$ is necessarily M\"obius bounded.
Therefore, Theorem \ref{thmexample} suggests the following question: If $T$ is a weakly homogeneous operator such that the intertwining operator $\Gamma(\varphi)$ is uniformly bounded, then does it follow that the operator  $T$ is necessarily similar to a homogeneous operator?

Moreover, if $T$ is similar to a homogeneous operator $S$ such that $\varphi(S)=U_\varphi S U_\varphi^*$ for some unitary representation $U$ of \mb, then the intertwining operators $\Gamma(\varphi)$, $\varphi\in\mb$, for $T$ are uniformly bounded in the sense that $\sup_{\varphi\in \mb}\|\Gamma(\varphi)\|$ is finite. Also, in this case, $\Gamma$ is a homomorphism from $\mb$ into $B(\hl)$. So, it is natural to ask the question: If $T$ is weakly homogeneous and the intertwiner $\Gamma$ is both a homomorphism and uniformly bounded, then does it follow that $T$ is similar to a homogeneous operator? This possibility was raised in \cite{onhomocontraction}.
These questions are also in the spirit of the three questions on similarity mentioned in \cite{pisier}.



\medskip \textbf{\textit{Acknowledgement}}.
The author would like to thank Prof. G. Misra for many fruitful discussions and
suggestions in the preparation of this paper. 

\end{document}